\documentclass{amsart}
\newcommand{\dd}

\usepackage{amsmath}
\usepackage[utf8]{inputenc}

\newcommand{\be}{\begin{equation}}
\newcommand{\ee}{\end{equation}}
\newcommand{\bea}{\begin{eqnarray}}
\newcommand{\eea}{\end{eqnarray}}

\newcommand{\bee}{\begin{eqnarray*}}
\newcommand{\eee}{\end{eqnarray*}}

 \newtheorem{thm}{Theorem}
 \newtheorem{cor}[thm]{Corollary}
 
 \newtheorem{lem}[thm]{Lemma}
 \newtheorem{prop}[thm]{Proposition}
 \theoremstyle{definition}
 \newtheorem{defn}[thm]{Definition}
 \newtheorem{fac}[thm]{Fact}
 \theoremstyle{remark}

\usepackage{enumitem}
 \usepackage{color}

\begin{document}

\title [$\mathrm{b}$-generalized skew derivations on multilinear polynomials in prime rings]{$\mathrm{b}$-generalized skew derivations acting as $2$-Jordan multiplier on multilinear polynomials  in prime rings}
\author[M. S. Pandey]{Mani Shankar Pandey }
\address{M. S. Pandey, The Institute of Mathematical Sciences, A CI of Homi Bhabha National Institute, C.I.T.Campus, Taramani, Chennai 600113, India.}
\email{manishankarpandey4@gmail.com}
\author[A. Pandey]{ Ashutosh Pandey }
\address{A. Pandey, School of Liberal Studies, Ambedkar University Delhi, Delhi-110006, INDIA.}
\email{ashutoshpandey064@gmail.com}


\begin{abstract}
	Let $\mathcal{R}$ be a prime ring of characteristic not equal to $2$, $\mathcal{U}$, $\mathcal{C}$ be its	Utumi quotient ring and extended centroid  respectively. Suppose $\phi(\zeta_1\ldots,\zeta_n)$  be a non central multilinear polynomial over $\mathcal{C}$,  and $\mathcal{F},\ \mathcal{G}$ be two $\mathrm{b}$-generalized skew derivations of $\mathcal{R}$ satisfying the identity $$\mathcal{F}(u)u+u \mathcal{G}(u)=\mathcal{G}(u^2), \ \text{for all $u\in\{\phi(\zeta)\ |\ \zeta=(\zeta_1\ldots,\zeta_n) \in \mathcal{R}^n\}$.}$$
 Then the purpose of this paper is to describe all possible forms of the $\mathrm{b}$-generalized skew derivations $\mathcal{F}$ and $\mathcal{G}$.   Consequently, we discuss the cases when this identity acts as Jordan derivation and $2$ - Jordan multiplier on prime rings.
	\end{abstract}
\maketitle
\section*{Introduction}
An associative ring $\mathcal{R}$ is said to be prime  if for any $x,y\in \mathcal{R}$, $x \mathcal{R} y= 0$, implies either $x = 0$ or $y = 0$. Throughout this research article, unless specifically stated, $\mathcal{R}$  will denote a  prime ring with center $\mathcal{Z}(\mathcal{R})$, $\mathcal{U}$ its Utumi quotient ring. Note that the quotient ring $\mathcal{U}$ is also a prime ring and its center $\mathcal{C}$ which is termed as the extended centroid of $\mathcal{R}$, is a field. The definition and construction of $\mathcal{U}$ can be found in \cite{beidar1995}. 

An additive mapping $\Delta:\mathcal{R}\rightarrow \mathcal{R}$ is said to be a derivation if $$\Delta(xy)=\Delta(x)y +x\Delta(y)$$ for all $x,y\in \mathcal{R}$. For a fixed $p \in \mathcal{R}$, the mapping $\Delta_p: \mathcal{R} \rightarrow \mathcal{R}$, defined by $\Delta_p (x ) = [p, x ]$ for all $x \in \mathcal{R}$ is a derivation, known as inner derivation induced by an element ${p}$. A derivation that is not inner is called outer derivation. Back in 1957, Posner \cite{posner1957} proved that if $\Delta$ is a non trivial derivation of a prime ring $\mathcal{R}$ such that $[\Delta(x), x] \in \mathcal{Z}(\mathcal{R})$ for all $x\in \mathcal{R}$ then $\mathcal{R}$ is a commutative ring. The results of Posner were generalized by mathematicians in several aspects. In 1993, Bre$\check{s}$ar \cite{brevsar1993centralizing} proved that if $d$ and $g$ are two non-zero derivations on the ring $\mathcal{R}$ satisfying $dx -x g\in \mathcal{Z}(\mathcal{R})$ for all $x\in \mathcal{R}$ then $\mathcal{R}$ is commutative.  Later on, many mathematicians extended these results on some appropriate subsets of a prime ring $\mathcal{R}$.

In 1991, M. Bre$\check{s}$ar \cite{Breaser1991} introduced a new type of derivation and termed it as generalized derivation. 
\begin{defn}
An additive mapping $\Phi:\mathcal{R}\rightarrow \mathcal{R}$ is said to be a generalized derivation if there exists a derivation $d$ on $\mathcal{R}$ such that $$\Phi(xy)=\Phi(x)y+xd(y),\ \text{for all $x,y \in \mathcal{R}$.}$$   
\end{defn}
For fixed $\alpha,\beta \in \mathcal{R}$, the mapping $\Phi_{(\alpha,\beta)}: \mathcal{R} \rightarrow \mathcal{R}$ defined by $\Phi_{(\alpha,\beta)}(x) = \alpha x+ x\beta$ is a generalized derivation on $\mathcal{R}$ and is usually termed as generalized inner derivation on $\mathcal{R}$. It can be visualized that every generalized derivation is a  derivation but the converse is not true in general. Thus the study of generalized derivation covers the concepts of both derivation and generalized derivation.

In 2011, Argaç and De Filippis \cite{AF2011} attempted and succeeded to give a generalization of  Posner’s theorem \cite{posner1957}. They describe the complete structure of the additive mappings satisfying the identity $\mathcal{F}(u)u -u\mathcal{G}(u) = 0$ for all $u=\phi(\zeta_1,\ldots,\zeta_n) \in  \mathcal{R}$, where $\phi$ is a  noncentral  multilinear polynomial identity on  $\mathcal{R}$, $\mathcal{F}$ and $\mathcal{G}$ are two generalized derivations on the prime ring $ \mathcal{R}$. 
\section*{$\mathrm{b}$-generalized skew derivation}
  The $\mathrm{b}$-generalized derivation is another generalization of generalized derivation. In 2014 Ko$\check{s}$an and Lee \cite{Kosan and Lee} have given the following definition:
  \begin{defn}
  Let $\mathcal{R}$ be a ring, $\mathrm{b}\in \mathcal{U}$ and $\Delta:\mathcal{R}\rightarrow \mathcal{U}$ be an additive map. Then the additive  mapping $\Gamma : \mathcal{R} \rightarrow \mathcal{U}$ associated to the pair $(\mathrm{b},\Delta)$ is called a $\mathrm{b}$-generalized derivation of $\mathcal{R}$ if
  $$\Gamma (xy) =\Gamma(x)y + \mathrm{b}x\Delta(y) \ \text{for all $x,y\in \mathcal{R}$}.$$ 
\end{defn}  
  The additive map $\Delta$ is uniquely determined by $\mathcal{F}$ and called an associated linear map of $\mathcal{F}$. Moreover, authors \cite{Kosan and Lee} have proved that if $\mathcal{R}$ is a prime ring and $\mathrm{b}\neq 0$ then the associated map $\Delta$ must be a derivation of $\mathcal{R}$. Here, we see that a $1$-generalized derivation is a generalized derivation. For some $\mathrm{p},\mathrm{q},\mathrm{b}\in  \mathcal{U}$, define a map $\Gamma_{\mathrm{p},\mathrm{q}}^{\mathrm{b}} : \mathcal{R} \rightarrow \mathcal{U}$ such that $\Gamma_{\mathrm{p},\mathrm{q}}^{\mathrm{b}}(x) = \mathrm{p} x + \mathrm{b} x\mathrm{q}$ for all $x \in \mathcal{R}$, is an example of $\mathrm{b}$-generalized derivation  and which is termed as inner $\mathrm{b}$-generalized derivation.  

The definition of generalized skew derivation is a unified notion of skew derivation and generalized derivation, which are considered classical linear mappings of non-commutative algebras. Note that a skew derivation on $\mathcal{R}$ associated with  the automorphism $\alpha$ is an additive mapping of $\mathcal{R}$ such that $\Delta(xy)=\Delta(x)y+ \alpha(x)\Delta(y)$ for all $x,y\in \mathcal{R}$.  A skew derivation associated with identity automorphism is a derivation and a generalized skew derivation associated with identity automorphism is a generalized derivation. 
\begin{defn}
An additive mapping $\psi$ on $\mathcal{R}$ is called generalized skew derivation associated with the automorphism $\alpha\ \in\ Aut(\mathcal{R})$ if there exists a skew derivation $\Delta$ on $\mathcal{R}$ such that $\psi (xy) =\psi(x)y + \alpha(x)\Delta(y)$ for all $x,y \in \mathcal{R}$. 
\end{defn}  
  
  In 2017, C. K. Liu \cite{C. K. Liu. 2} generalized the result of Posner \cite{posner1957} by taking $\mathrm{b}$-generalized derivation with Engel conditions on the prime ring $\mathcal{R}$.
  
  In a study of  different types of derivations, in 2018, De Filippis and Wei \cite{de2018} introduced the notion of $\mathrm{b}$-generalized skew derivation.
\begin{defn}
Let $\mathcal{R}$ be an associative  ring,  $\mathrm{d}:\mathcal{R}\rightarrow \mathcal{R}$ be an additive mapping, $\alpha$ be an automorphism of $\mathcal{R}$ and $\mathrm{b}$ be a fixed element in the Utumi quotient ring $\mathcal{U}$. Then an additive mapping $\mathcal{F} : \mathcal{R} \rightarrow \mathcal{R}$, is said to be a $\mathrm{b}$-generalized skew derivation of $\mathcal{R}$ associated with the triplet  $(\mathrm{b},\alpha,\mathrm{d})$ if $$\mathcal{F}(xy) = \mathcal{F}(x)y+ \mathrm{b}\alpha(x)\mathrm{d}(y),\ \text{for all $x,y\in \mathcal{R}$.}$$ 
\end{defn}
Moreover, the authors have proved that if the ring $\mathcal{R}$ is prime and $\mathrm{b}\neq 0$ then  the associated additive map $\mathrm{d}$, defined above, is a skew derivation. Further, it has been shown that the additive mapping  $\mathcal{F}$ can be extended to the Utumi quotient ring $\mathcal{U}$ and it assumes the form $\mathcal{F}(x)= \mathrm{a}x+ \mathrm{b}\mathrm{d}(x)$, where $\mathrm{a} \in \mathcal{U}$. The concept of  $\mathrm{b}$-generalized skew derivation with the associated term $(\mathrm{b}, \alpha, \mathrm{d})$ covers the concepts of skew derivation, generalized derivations, and left multipliers, etc.  For instance if we choose $\mathrm{b}= 1$, $\mathrm{b}$-generalized skew derivation becomes a skew derivation and for $\mathrm{b} = 1, \alpha=I_{\mathcal{R}}$, $\mathrm{b}$-generalized skew derivation becomes a generalized derivation, where $I_{\mathcal{R}}$ is the identity map on the ring $\mathcal{R}$. Further, if we take $\mathrm{b}=0$ then $\mathrm{b}$-generalized skew derivation becomes a left multiplier map.  Let $\mathcal{R}$ be a prime ring and $\alpha$ be an automorphism of
$\mathcal{R}$. Then the mapping $\mathcal{F} : \mathcal{R}\longrightarrow \mathcal{R}$ given by $ x \mapsto  ax + \mathrm{b}\alpha(x)c$ is an $\mathrm{b}$-generalized skew derivation of $\mathcal{R}$ with associated triplet $(\mathrm{b},\alpha,\mathrm{d})$, where $a, b$ and $c$ are fixed elements in $\mathcal{R}$ and $d(x) = \alpha(x)c - cx$, for all $x \in \mathcal{R}$. Such $\mathrm{b}$-generalized skew derivation is called as inner $\mathrm{b}$-generalized skew derivation.

Recently, De Filippis and Wei \cite{de 2018} studied an identity related to $\mathrm{b}$-generalized skew derivations on the prime ring $\mathcal{R}$ with multilinear polynomial over $\mathcal{C}$. More precisely, they proved the following.
\begin{thm}
Let $\mathcal{R}$ be a prime ring of characteristic different from 2, $\mathcal{Q}_r$ be its right
Martindale quotient ring and $\mathcal{C}$ be its extended centroid, $\mathcal{G}$ a nonzero $X$-generalized
skew derivation of $\mathcal{R}$, and $\mathcal{S}$ be the set of the evaluations of a multilinear polynomial
$f(\chi_1,\ldots,\chi_n)$ over $\mathcal{C}$ with $n$ non-commuting variables. Let $u,v\in\mathcal{R}$ be such that $u\mathcal{G}(\chi)\chi + \mathcal{G}(\chi)\chi v = 0$ for all $\chi \in \mathcal{S}$. Then one of the following statements holds:
\begin{enumerate}
    \item  $v\in \mathcal{C}$ and there exist $a, b, c \in \mathcal{Q}_r$ such that $\mathcal{G}(\chi) = a\chi + b\chi c$ for all $\chi\in \mathcal{R}$ with
$(u + v)a = (u + v)b = 0$;
\item $f(\chi_1,\ldots,\chi_n)^2$  is central-valued on $\mathcal{R}$  and there exists $a
\in \mathcal{Q}_r$ such that $\mathcal{G}(\chi) =a\chi$ for all $\chi\in \mathcal{R}$ with $ua +av =0$.
\end{enumerate}
\end{thm}
In \cite{V. D. Filippis 2021}, Filippis et. al. continue this line of investigation  by studying  the identity $\mathcal{F}(u)u - u\mathcal{G}(u) = 0$ for all $u\in\{\phi(\zeta)\ |\ \zeta=(\zeta_1\ldots,\zeta_n) \in \mathcal{R}^n\}$ and describe the complete structure of the $\mathrm{b}$-generalized skew derivations $\mathcal{F}$ and $\mathcal{G}$, where $\phi(\zeta)$ is a multilinear polynomial and $\mathcal{\zeta}=(\mathcal{\zeta}_1,\ldots,\mathcal{\zeta}_n)\in \mathcal{R}^n$. Many generalizations which are relevant to our discussion about $\mathrm{b}$-generalized skew derivation can be found in \cite{de2018, de 2018, V. D. Filippis 2021}. 
\section*{Main Theorem}
This will be an interesting problem when $\mathrm{b}$-generalized skew derivations act as a Jordan derivation or as a $n$-Jordan multiplier. Following this line of investigation, our main theorem presents a detailed description of the maps $\mathcal{F}$ and $\mathcal{G}$, in the cases when $\mathcal{F}$ acts as a $2$-Jordan derivation and $\mathcal{G}$  acts as a $2$-Jordan multiplier. Here, we study the $\mathrm{b}$-generalized skew derivation identity $\mathcal{F}(u)u+u\mathcal{G}(u) = \mathcal{G}(u^2)$ for all $u\in\{\phi(\zeta)\ |\ \zeta=(\zeta_1\ldots,\zeta_n) \in \mathcal{R}^n\}$, where $\phi(\zeta)$ is a multilinear polynomial and $\mathcal{F}$, $\mathcal{G}$ are $\mathrm{b}$-generalized skew derivations. More precisely, the statement of our main theorem is the following.
\begin{thm} \label{thm}
Let $\mathcal{R}$  be prime ring with $char(\mathcal{R})\neq 2$, $\mathcal{U}$ be the Utumi quotient ring of $\mathcal{R}$ and $\mathcal{C}$ be the extended centroid of $\mathcal{R}$, $\mathcal{F}$ and $\mathcal{G}$ be two $\mathrm{b}$-generalized skew derivations on $\mathcal{R}$ with associated triplets $(\tilde{b}, \alpha, g)$ and $(b^{\prime}, \alpha, h)$ respectively. Let $\phi(\zeta)$ be a non central polynomial identity of $\mathcal{R}$ such that $$\mathcal{F}(u)u+u \mathcal{G}(u)=\mathcal{G}(u^2)$$ for all $u\in\{\phi(\zeta)|\zeta=(\zeta_1\ldots,\zeta_n) \in \mathcal{R}^n\}$. Then one of the following holds:
		\begin{enumerate}
	    \item There exist $a,\tilde{a},c,\tilde{c} \in \mathcal{U}$ such that $\mathcal{F}(\chi)= a\chi+\chi\tilde{a}$ and $\mathcal{G}(\chi)=c\chi+\chi\tilde{c}$ for all $\chi \in \mathcal{R}$ with $  \tilde{a}+c \in \mathcal{C}$ and $\tilde{a}+a=0$;
	    \item There exist $a,\tilde{a},\hat{c} \in \mathcal{U}$ such that $\mathcal{F}(\chi)= a\chi+\chi\tilde{a}$ and $\mathcal{G}(\chi)=\hat{c}\chi$ for all $\chi \in \mathcal{R}$ with $\tilde{a}+\hat{c}\in \mathcal{C}$ and $\tilde{a}+a=0$;
     \item There exist $c\in \mathcal{C}$ and $\tilde{c}\in \mathcal{U}$ such that $\mathcal{F}(\chi)= 0$ and $\mathcal{G}(\chi)=c\chi+\chi\tilde{c}$ for all $\chi \in \mathcal{R}$.
	    \item There exist $\tilde{c}\in \mathcal{C}$ such that $\mathcal{F}(\chi)= 0$ and $\mathcal{G}(\chi)=\tilde{c}\chi$ for all $\chi \in \mathcal{R}$.
	    \item There exist   $\tilde{b},b^{\prime},c\in \mathcal{U}$ and $\eta,b^{\prime}\eta-c\in \mathcal{C}$ such that $\mathcal{F}(\chi)=b^{\prime}h(\chi)+[\tilde{b}\eta ,\chi]$ and $\mathcal{G}(\chi)=c\chi+b^{\prime}h(x)$, where $h$ is a skew derivation associated with the automorphism(inner) $\alpha$ satisfying $\tilde{b}\alpha(\chi)=\chi\tilde{b},$ for all $\chi\in \mathcal{R}$.
     	\end{enumerate}
	\end{thm}
\section*{Preliminaries}
In this section, we recall some preliminary results which will be used frequently throughout the paper.
\begin{fac}\label{fac2}
	Let $\mathcal{R}$ be a prime ring and $\mathcal{I}$ a two-sided ideal of $\mathcal{R}$. Then $\mathcal{R}$, $\mathcal{I}$ and $\mathcal{U}$ satisfy the same generalized polynomial identities with coefficients in $\mathcal{U}$ \cite{beidar1978ri}.
\end{fac}

\begin{fac}\label{fac3}
	Let $\mathcal{R}$ be a prime ring and $\mathcal{I}$ a two-sided ideal of $\mathcal{R}$. Then $\mathcal{R}, \mathcal{I}$ and $\mathcal{U}$
	satisfy the same differential identities \cite{lee1992semiprime}.
\end{fac}
\begin{fac}\label{fac4}
	Let $\mathcal{R}$ be a prime ring. Then every derivation $d$ of $\mathcal{R}$ can be uniquely
	extended to a derivation of $\mathcal{U}$ \cite{beidar1978ri}.
\end{fac}
\begin{fac}\label{fac5}
	(Kharchenko [Theorem 2,\cite{kharchenko1978diff}]) Let $\mathcal{R}$ be a prime ring, $d$ a nonzero derivation on $\mathcal{R}$ and $\mathcal{I}$ a nonzero ideal of $\mathcal{R}$. If $\mathcal{I}$ satisfies the differential identity.
	\begin{equation*}
			\phi\big(\zeta_1,\zeta_2,\ldots,\zeta_n,d(\zeta_{1}),d(\zeta_{2})\ldots, d({\zeta_n})\big)=0
		\end{equation*}
	for any $\zeta_1,\ldots,\zeta_n\in \mathcal{I}$, then either
	\begin{itemize}
		\item $\mathcal{I}$ satisfies the generalized polynomial identity
		\begin{equation*}
			\phi(\zeta_1,\zeta_2\ldots,\zeta_n,\chi_1, \chi_2\ldots, \chi_n)=0
			\end{equation*} for all $\chi_1, \chi_2\ldots, \chi_n \in \mathcal{R}$.\\
		or
		\item $d$ is $\mathcal{U}$-inner i.e., for some $\mathcal{P}\in \mathcal{U}$, $d(\zeta) = [\mathcal{P}, \zeta]$ and $\mathcal{I}$ satisfies the generalized polynomial identity.
		\begin{equation*}
			\phi(\big(\zeta_1,\zeta_2,\ldots,\zeta_n,[\mathcal{P},\zeta_1],[\mathcal{P},\zeta_2]\ldots, [\mathcal{P},\zeta_n]\big)=0
		\end{equation*}
	\end{itemize}
\end{fac}
\begin{fac}\label{fac6}
	Let $\mathcal{\zeta}=\{\zeta_1,\zeta_2,\ldots\}$ represents a countable set  of non-commuting indeterminates $\zeta_1,\zeta_2,\ldots$. Let $\mathcal{C}\{\mathcal{\zeta}\}$ denotes the free algebra over $\mathcal{C}$ on the set $\mathcal{\zeta}$ and $\mathcal{T}=\mathcal{U}\ast_\mathcal{C}\mathcal{C}\{\mathcal{\zeta}\}$, denotes the free product of the $\mathcal{C}$-algebras $\mathcal{U}$ and $\mathcal{C}\{\mathcal{\zeta}\}$. The members of $\mathcal{T}$ are known as the generalized polynomials with coefficients in $\mathcal{U}$. Let $\mathcal{B}$ be a set of $\mathcal{C}$-independent vectors of $\mathcal{U}$. Then any $\mathcal{G} \in \mathcal{T}$ can be expressed in the form $\mathcal{G}=\sum_i \beta_i \mathcal{V}_i$, where $\beta_i\in \mathcal{C}$ and $\mathcal{V}_i$ are $\mathcal{B}$-monomials of the form $a_0\xi_1a_1\xi_2a_2\cdots \xi_n a_n$, with $a_0, a_1,\ldots, a_n\in \mathcal{B}$ and $\xi_1, \xi_2,\ldots, \xi_n\in \chi$. Any generalized polynomial $g=\sum_i \beta_i \mathcal{V}_i$ is trivial i.e., zero element in $\mathcal{T}$ if and only if $\beta_i=0$ for each $i$. Further details can be found in \cite{C. L. Chuang. 1988}. If each  monomial of a generalized polynomial $\phi(\zeta_1,\ldots,\zeta_n)$ contains each $\zeta_i$ only once for $1\leq i\leq n$, then $\phi(\zeta_1,\ldots,\zeta_n)$ is said to be multilinear polynomial.
\end{fac}
\begin{fac}\label{fac9}
	Let $\mathcal{K}$ be an infinite field and $m\geq 2$ an integer. If $\mathcal{P}_1,\ldots,\mathcal{P}_k$ are non-scalar matrices in $\mathcal{M}_m(\mathcal{K})$ then there exists some invertible matrix $\mathcal{P} \in \mathcal{M}_m(\mathcal{K})$ such that each matrix $\mathcal{P}\mathcal{P}_1\mathcal{P}^{-1},\ldots,\mathcal{P}\mathcal{P}_k\mathcal{P}^{-1}$ has all non-zero entries. \cite{de2012}
\end{fac}
\begin{fac}\label{fac 8}
Let $\mathcal{R}$ be a non-commutative prime ring  with $Char(R) \neq 2$, Utumi quotient ring $\mathcal{U}$ and extended centroid $\mathcal{C}$. Also let  $\phi(\zeta_1,\ldots,\zeta_n)$ be a multilinear polynomial over $\mathcal{C}$ which is not central valued on $\mathcal{R}$ and $a,b\in \mathcal{R}$ such that
$$a\phi(\zeta)^2+\phi(\zeta)b\phi(\zeta)=0, \forall \ \zeta=(\mathcal{\zeta}_1,\ldots,\mathcal{\zeta}_n)\in \mathcal{R}^n.$$
Then $a=-b\in  \mathcal{C}$ [\cite{V. D. Filippis 2021}, Corollary 1].
\end{fac}
\begin{fac}\label{fac 9}
If $d$ is a non-zero skew derivation on a prime ring $\mathcal{R}$ then associated automorphism $\alpha$ is unique.
\end{fac}
Let $\phi^ d (\zeta_1,\ldots,\zeta_n)$ be the polynomial originated from $\phi(\zeta_1,\ldots,\zeta_n)$ by replacing
each coefficient $\gamma_{\tau}$ with $d(\gamma_{\tau})$. We notice that
\begin{align*}
    d(\gamma_{\tau}\zeta_{\tau(1)},\ldots,\zeta_{\tau(n)} )= d(\gamma_{\tau})\zeta_{\tau(1)},\ldots,\zeta_{\tau(n)} \\
    +\gamma_{\tau}\sum\limits_{i=0}^{n-1}\zeta_{\tau(1)}\ldots ,\ldots,d(\zeta_{\tau(j)}),\zeta_{\tau(j+1)}\ldots,\zeta_{\tau(n)}
\end{align*}
and 
\begin{align*}
d(\phi(\zeta_1,\ldots,\zeta_n) )= \phi^d (\zeta_1,\ldots,\zeta_n) +\sum\limits_{\tau\in \mathcal{S}_n}\gamma_{\tau}\sum\limits_{i=0}^{n-1}\zeta_{\tau(1)}\ldots ,\ldots,d(\zeta_{\tau(j)}),\zeta_{\tau(j+1)}\ldots,\zeta_{\tau(n)}
\end{align*}
Let $\mathbb{Z}\langle \zeta_1,\ldots,\zeta_k\ldots\rangle$ be the free algebra on the countable set $(\zeta_1,\ldots,\zeta_k\ldots)$ over the set of integers $\mathbb{Z}$ and $\phi(\zeta_1,\ldots,\zeta_n)$ be a polynomial such that at least one of its monomials of highest degree has coefficient $1$. Let $\mathcal{S}$ be a nonempty subset of a ring $\mathcal{R}$. We say that $\phi$ is a polynomial identity on  $\mathcal{S}$  if  $\phi(\zeta_1,\ldots,\zeta_n)=0$ for all $\zeta_1,\ldots,\zeta_n\in \mathcal{S}$. Now we give the definition of multilinear polynomials.
\begin{defn}
 A polynomial  $\phi(\zeta_1,\ldots,\zeta_k,\ldots )\in \mathbb{Z}\langle \zeta_1,\ldots,\zeta_k,\ldots \rangle$ is said to be  multilinear if every  $\zeta_i , i=1,2,\ldots,n$ appears exactly once in each of the monomials of $\phi$.
\end{defn}
Note that in this article $\mathcal{R}$ always denotes a non-trivial and associative prime ring (unless otherwise stated) and we use the abbreviation GPI for generalized polynomial identity. 
\section*{Inner Case}
In this section, we study the case when $\mathcal{F}$ and $\mathcal{G}$ are $b$-generalized skew inner derivations. Suppose $\mathcal{G}(\chi) = c\chi+b^{\prime}\alpha(\chi)b$ and $\mathcal{F}(\chi) = a\chi+\tilde{b}\alpha(\chi)d$, for all $\chi \in \mathcal{R}$ and for some $a, b, c, d, b^{\prime}, \tilde{b} \in \mathcal{U}$ and $\alpha$ is an automorphism of $\mathcal{R}$ then we prove the following Proposition:

\begin{prop}\label{prop1}
	Let $\mathcal{R}$ be a prime ring of characteristic not equal to $2$ with Utumi quotient ring $\mathcal{U}$, extended centroid $\mathcal{C}$ and $\phi(\zeta_1,\ldots,\zeta_n)$ be a non central multilinear polynomial over $\mathcal{C}$. Let $\mathcal{G}$ and $\mathcal{F}$ be two $b$-generalized skew inner derivations on $\mathcal{R}$ with associated terms $(b^{\prime}, \alpha)$,  and  $(\tilde{b},\alpha)$ respectively such that $$\mathcal{F}(u)u+u \mathcal{G}(u)=\mathcal{G}(u^2)$$ for all $u\in\{\phi(\zeta)|\zeta=(\zeta_1\ldots,\zeta_n) \in \mathcal{R}^n\}$. Then one of the following holds:
		\begin{enumerate}
	    \item There exist $a,\tilde{a},c,\tilde{c} \in \mathcal{U}$ such that $\mathcal{F}(\chi)= a\chi+\chi\tilde{a}$ and $\mathcal{G}(\chi)=c\chi+\chi\tilde{c}$ for all $\chi \in \mathcal{R}$ with $  \tilde{a}+c \in \mathcal{C}$ and $\tilde{a}+a=0$.
	    \item There exist $a,\tilde{a},\hat{c} \in \mathcal{U}$ such that $\mathcal{F}(\chi)= a\chi+\chi\tilde{a}$ and $\mathcal{G}(\chi)=\hat{c}\chi$ for all $\chi \in \mathcal{R}$ with $\tilde{a}+\hat{c}\in \mathcal{C}$ and $\tilde{a}+a=0$.
	   \item There exist $c\in \mathcal{C}$ and $\tilde{c}\in \mathcal{U}$ such that $\mathcal{F}(\chi)= 0$ and $\mathcal{G}(\chi)=c\chi+\chi\tilde{c}$ for all $\chi \in \mathcal{R}$.
	    \item There exist $\tilde{c}\in \mathcal{C}$ such that $\mathcal{F}(\chi)= 0$ and $\mathcal{G}(\chi)=\tilde{c}\chi$ for all $\chi \in \mathcal{R}$.
	  	   	      	    \end{enumerate}
	\end{prop}
To prove the above-mentioned proposition, we need the following lemmas.
\begin{lem}\label{lem2}
   Let $\mathcal{R} = \mathcal{M}_d(\mathbb{K})$ be the ring of all $d\times d$ matrices over the field $\mathbb{K}$ with $Char(\mathbb{K})\neq 2,\ d \geq 2$ and $f(\chi_1,\ldots,\chi_n)$ be a non central multilinear polynomial over $\mathbb{K}$. Let $a_1,a_2, \ldots,a_6 \in \mathcal{U}$ such that $ a_1f(\zeta)^2 +a_2f(\zeta)a_3f(\zeta)+f(\zeta)a_5f(\zeta)a_6+f(\zeta)a_4f(\zeta)-a_5f(\zeta)^2a_6=0$ for all $\mathcal{\zeta} =(\mathcal{\zeta}_1,\ldots,\mathcal{\zeta}_n)\in \mathcal{R}^{n}$. Then one of the following holds:
\begin{enumerate}
    \item $a_2,a_5 \in \mathbb{K}.I_{d}$.
    \item $a_2, a_6 \in \mathbb{K}.I_{d}$.
    \item $a_3,a_5 \in \mathbb{K}.I_{d}$.
    \item $a_3, a_6 \in \mathbb{K}.I_{d}$.
\end{enumerate}
\end{lem}
\begin{proof}
From the hypothesis $\mathcal{R}$ satisfies the following generalized polynomial identity:
\begin{align}\label{eq1a}
    h(\mathcal{\zeta}_1,\ldots,\mathcal{\zeta}_n) = a_1f(\zeta)^2 +a_2f(\zeta)a_3f(\zeta)+f(\zeta)a_5f(\zeta)a_6+f(\zeta)a_4f(\zeta)-a_5f(\zeta)^2a_6
\end{align}
for all $\mathcal{\zeta} =(\mathcal{\zeta}_1,\ldots,\mathcal{\zeta}_n)\in \mathcal{R} ^{n}$. 

\textbf{Case 1:} Suppose $\mathbb{K}$ is infinite. At first glance, we show that either $a_5$ or $a_6$ is central. Assume that none of $a_5,a_6$ is central. Then we shall prove that this case leads to a contradiction.
From Fact \ref{fac9}, there exists an automorphism $\rho\in Aut(\mathcal{R})$ such that $\rho(a_5)=a_5'$ and $\rho(a_6)=a_6'$, have all non-zero entries. Now since the Equation (\ref{eq1a}) is invariant under the action of the automorphism $\rho$, $\mathcal{R}$ satisfies
\begin{align}\label{eq1b}
a_1'f(\zeta)^2 +a_2'f(\zeta)a_3'f(\zeta)+f(\zeta)a_5'f(\zeta)a_6'+f(\zeta)a_4'f(\zeta)-a_5'f(\zeta)^2a_6'=0
\end{align}
where $\rho(a_l)=a_l^{\prime}$,  $1\leq l\leq 6$. Since $f(\zeta), \ \zeta=(\mathcal{\zeta}_1,\ldots,\mathcal{\zeta}_n)$ is non central, by \cite{lee1992semiprime} there exists a sequence of matrices $\zeta=(\mathcal{\zeta}_1,\ldots,\mathcal{\zeta}_n) \in \mathcal{R}$ such that $f (\zeta) = \mathcal{A}e_{ij}$ with non-zero $\mathcal{A}\in \mathbb{K}$ and $i \neq j$. Here $e_{ij}$ denotes the matrix whose $(i, j)$-entry is $1$, and the rest entries are zero. Since $f (\mathcal{R}) = \{f (\chi_1 ,\ldots, \chi_n ) :\chi_i \in \mathcal{R}\}$ is invariant under the action of all inner automorphisms of $\mathcal{\mathcal{R}}$, , for $i\neq j$, there exists a sequence of matrices $\zeta=(\mathcal{\zeta}_1,\ldots,\mathcal{\zeta}_n)$ in $\mathcal{R}$ such that $f (\mathcal{\zeta}_1,\ldots,\mathcal{\zeta}_n) = e_{ij}$. Thus from Equation (\ref{eq1b}), we get
\begin{align}\label{eq1c}
a_2'e_{ij}a_3'e_{ij}+e_{ij}a_5'e_{ij}a_6'+e_{ij}a_4'e_{ij}=0
\end{align}
Now right multiplying by $e_{ij}$ in Equation (\ref{eq1c}), we get
\begin{align*}
e_{ij}a_5'e_{ij}a_6'e_{ij}=(a_5')_{ji}(a_6')_{ji}e_{ij}=0
\end{align*}
It implies either $(a_5')_{ji} = 0$ or $(a_6')_{ji} = 0$. In either case, we have a contradiction.  Thus  either $a_5\in \mathbb{K}.I_{d}$ or $a_6\in \mathbb{K}.I_{d}$. Now if $a_5\in  \mathbb{K}.I_{d}$ then Equation (\ref{eq1a}) reduces to
\begin{align*}
     a_1f(\zeta)^2 +a_2f(\zeta)a_3f(\zeta)+f(\zeta)a_4f(\zeta)=0
\end{align*}
for all $\mathcal{\zeta} =(\mathcal{\zeta}_1,\ldots,\mathcal{\zeta}_n)\in \mathcal{R} ^{n}$. Again by parallel arguments we can show that either $a_2\in \mathbb{K}.I_{d}$ or $a_3\in \mathbb{K}.I_{d}$.  Similarly, if $a_6\in  \mathbb{K}.I_{d}$ then we can show that either $a_2\in \mathbb{K}.I_{d}$ or $a_3\in \mathbb{K}.I_{d}$. Combining all the above outcomes we get our required results.
	
\textbf{Case 2:}	Suppose the field  $\mathbb{K}$ is finite. Let $\mathbb{L}$ be an infinite field such that $\mathbb{K} \subseteq \mathbb{L}$. Let $\bar{\mathcal{R}} =\mathcal{M}_d(\mathbb{L}) = \mathcal{R} \otimes_\mathbb{K} \mathbb{L}$. This is easy to see that a multilinear polynomial $f(\chi_1,\ldots,\chi_n)$ is central valued on $\mathcal{R}$ if and only if it is central valued on $\bar{\mathcal{R}}$.  Since $h(\mathcal{\zeta}_1,\ldots,\mathcal{\zeta}_n) $ is GPI on $\mathcal{R}$ and it is multi-homogeneous of multi degree $(2, \ldots , 2)$ in the indeterminates $\chi_1,\ldots,\chi_n$. Hence the complete linearization of $h(\mathcal{\zeta}_1,\ldots,\mathcal{\zeta}_n) $ is a multilinear generalized polynomial identity $\Phi(\mathcal{P}_1,\ldots, \mathcal{P}_n, \mathcal{\zeta}_1,\ldots,\mathcal{\zeta}_n)$ in $2n$ indeterminate. Also, $\Phi(\mathcal{\zeta}_1,\ldots,\mathcal{\zeta}_n, \mathcal{\zeta}_1,\ldots,\mathcal{\zeta}_n) = 2^nh(\mathcal{\zeta}_1,\ldots,\mathcal{\zeta}_n)$. Since $\Phi(\mathcal{P}_1,\ldots, \mathcal{P}_n,$\\ $\mathcal{\zeta}_1,\ldots,\mathcal{\zeta}_n)$ is generalized polynomial identity for both $\mathcal{R}$, $\bar{\mathcal{R}}$ and $Char(\mathcal{R})$ is different from $2$, we obtain $h(\mathcal{\zeta}_1,\ldots,\mathcal{\zeta}_n)=0 $ for all  $\zeta_1,\ldots,{\zeta}_n \in \bar{\mathcal{R}}$. Thus our assertion follows from Case $1$.
\end{proof}

\begin{lem} \label{lem3}
    Let $\mathcal{R}$ be a non-commutative prime ring with characteristics different from $2$. Let $\mathcal{U}$ be the Utumi ring of quotients and $\mathcal{C}$ be the extended centroid of the ring $\mathcal{R}$. Suppose $f(\mathcal{\zeta}_1,\ldots,\mathcal{\zeta}_n)$ be a non central multilinear polynomial over $\mathcal{C}$. Let $a_1,a_2, \ldots,a_6 \in \mathcal{U}$ such that $ a_1f(\zeta)^2  +a_2f(\zeta)a_3f(\zeta)+f(\zeta)a_5f(\zeta)a_6+f(\zeta)a_4f(\zeta)-a_5f(\zeta)^2a_6=0$ for all $\zeta =(\mathcal{\zeta}_1,\ldots,\mathcal{\zeta}_n)\in \mathcal{R}^{n}$. If $\mathcal{R}$ does not satisfy any nontrivial generalized
polynomial identity then one of the following holds:
	 \begin{enumerate}
    \item $a_2,a_5 \in \mathcal{C}$.
    \item $a_2, a_6 \in \mathcal{C}$.
    \item $a_3,a_5 \in \mathcal{C}$.
    \item $a_3, a_6 \in \mathcal{C}$.
    \end{enumerate}
\end{lem}	
\begin{proof}
First, we will prove that  both $a_5$ and $a_6$ are central. Suppose on the contrary that neither $a_5$ nor $a_6$ is central. From the hypothesis we have
\begin{align}\label{eq3a}
  h(\mathcal{\zeta}_1,\ldots,\mathcal{\zeta}_n) = a_1f(\mathcal{\zeta}_1,\ldots,\mathcal{\zeta}_n)^2 +a_2f(\mathcal{\zeta}_1,\ldots,\mathcal{\zeta}_n)a_3f(\mathcal{\zeta}_1,\ldots,\mathcal{\zeta}_n)  \nonumber\\ 
    -a_5f(\mathcal{\zeta}_1,\ldots,\mathcal{\zeta}_n)^2a_6 +f(\mathcal{\zeta}_1,\ldots,\mathcal{\zeta}_n)a_4f(\mathcal{\zeta}_1,\ldots,\mathcal{\zeta}_n) \nonumber\\
    +f(\mathcal{\zeta}_1,\ldots,\mathcal{\zeta}_n)a_5f(\mathcal{\zeta}_1,\ldots,\mathcal{\zeta}_n)a_6
\end{align}
for all $\mathcal{\zeta}_1,\ldots,\mathcal{\zeta}_n\in \mathcal{R}$. Let $\mathcal{T} = \mathcal{U}\star_\mathcal{C} \mathcal{C}\{\chi_1,\ldots,\chi_n\}$ be the free product of the Utumi quotient ring $\mathcal{U}$ and the free $\mathcal{C}$ algebra  $\mathcal{C}\{\chi_1,\ldots,\chi_n\}$ in non commuting variables $\chi_1,\ldots,\chi_n$. Since $\mathcal{R}$ and $\mathcal{U}$ satisfy the same generalized polynomial identity (GPI) (for details see \cite{C. L. Chuang. 1988}), thus  $\mathcal{U}$ satisfies $h(\chi_1, \ldots, \chi_n ) =0_\mathcal{T}$. Since $\{1,a_6\}$ is linearly $\mathcal{C}$-independent thus from Fact \ref{fac6}, we get
\begin{equation}
   -a_5f(\mathcal{\zeta}_1,\ldots,\mathcal{\zeta}_n)^2a_6+ f(\chi_1 , \ldots, \chi_n )a_5f(\chi_1 , \ldots, \chi_n )a_6=0
\end{equation}for all $\mathcal{\zeta}_1,\ldots,\mathcal{\zeta}_n\in \mathcal{R}$. Again since $\{1,a_5\}$ is linearly $\mathcal{C}$-independent hence from Fact \ref{fac6}, we get $f(\mathcal{\zeta}_1,\ldots,\mathcal{\zeta}_n)a_5f(\mathcal{\zeta}_1,\ldots,\mathcal{\zeta}_n)a_6=0$, a contradiction. Thus either $a_5\in \mathcal{C}$ or $a_6\in \mathcal{C}$. \\
\textbf{Case 1:} Assume $a_5\in \mathcal{C}$. Then since $a_6 \notin \mathcal{C}$ thus from Fact \ref{fac6}, we get 
\begin{equation} \label{eq3ba}
  a_1f\mathcal{\zeta}_1,\ldots,\mathcal{\zeta}_n)^2 +a_2f(\mathcal{\zeta}_1,\ldots,\mathcal{\zeta}_n)a_3f(\mathcal{\zeta}_1,\ldots,\mathcal{\zeta}_n) \end{equation}\begin{equation*}
  +f\mathcal{\zeta}_1,\ldots,\mathcal{\zeta}_n)a_4f\mathcal{\zeta}_1,\ldots,\mathcal{\zeta}_n) =0
\end{equation*}
 for all $\mathcal{\zeta}_1,\ldots,\mathcal{\zeta}_n\in \mathcal{R}$. Now suppose that $a_2,a_3$ are not central elements. Then $\{1,a_3,a_4\}$ is linearly $\mathcal{C}$-dependent otherwise $a_2f(\mathcal{\zeta}_1,\ldots,\mathcal{\zeta}_n)a_3f(\mathcal{\zeta}_1,\ldots,\mathcal{\zeta}_n)=0$ will appear as a non-trivial polynomial identity, a contradiction. Hence there exist $\beta_1,\beta_2, \beta_3 \in \mathcal{C}$ such that $\beta_1+\beta_2 a_3+\beta_3 a_4=0$. If $\beta_3 =0$ then $a_3= -\beta_1\beta_2^{-1}$, a contradiction. Hence $\beta_3 \neq 0$, which implies $a_4= \gamma_1+\gamma_2 a_3$, where $\gamma_1= -\beta_1\beta_3^{-1}$ and $\gamma_2= -\beta_2\beta_3^{-1}$. Thus Equation \ref{eq3ba} reduces to
 \begin{equation} 
  a_1f\mathcal{\zeta}_1,\ldots,\mathcal{\zeta}_n)^2 +(a_2+\gamma_1)f(\mathcal{\zeta}_1,\ldots,\mathcal{\zeta}_n)a_3f(\mathcal{\zeta}_1,\ldots,\mathcal{\zeta}_n) =0.
\end{equation} Since $\{1,a_3\}$ is linearly $\mathcal{C}$-independent thus from Fact 11, $\mathcal{R}$ satisfies the non-trivial polynomial identity $(a_2+\gamma_1)f(\mathcal{\zeta}_1,\ldots,\mathcal{\zeta}_n)a_3f(\mathcal{\zeta}_1,\ldots,\mathcal{\zeta}_n) =0$, a contradiction. Hence either $a_2$ or $a_3 \in \mathcal{C}$. Thus in this case we get the Conclusion (1) and (3).\\
\textbf{Case 2:} Assume $a_6 \in \mathcal{C}$. Again in this case by using similar arguments as above we get either $a_2$ or $a_3$ is central. Thus in this case we get the Conclusion (2) and (4).
\end{proof}
\begin{lem}\label{lem4}
    Let $\mathcal{R}$ be a non-commutative prime ring with characteristic different from $2$. Let $\mathcal{U}$ be the Utumi ring of quotients and $\mathcal{C}$ be the extended centroid of the ring $\mathcal{R}$. Suppose $f(\mathcal{\zeta}_1,\ldots,\mathcal{\zeta}_n)$ be a non central multilinear polynomial over $\mathcal{C}$. Let $a_1,a_2, \ldots,a_6 \in \mathcal{U}$ such that $ a_1f(\zeta)^2 +a_2f(\zeta)a_3f(\zeta)+f(\zeta)a_5f(\zeta)a_6+f(\zeta)a_4f(\zeta)-a_5f(\zeta)^2a_6=0$ for all $\zeta=(\mathcal{\zeta}_1,\ldots,\mathcal{\zeta}_n)\in R^{n}$. Then one of the following holds:
	 \begin{enumerate}
    \item $a_2,a_5 \in \mathcal{C}$.
    \item $a_2, a_6 \in \mathcal{C}$.
    \item $a_3,a_5 \in \mathcal{C}$.
    \item $a_3, a_6 \in \mathcal{C}$.
    \end{enumerate}
\end{lem}	
\begin{proof}
    It is obvious from Lemma \ref{lem3} that the prime ring $\mathcal{R}$  satisfies the following non-trivial
generalized polynomial identity:
\begin{align}\label{eq4a}
    h(\mathcal{\zeta}_1,\ldots,\mathcal{\zeta}_n) = a_1f(\mathcal{\zeta}_1,\ldots,\mathcal{\zeta}_n)^2 +a_2f(\mathcal{\zeta}_1,\ldots,\mathcal{\zeta}_n)a_3f(\chi_1 , \ldots, \chi_n )  \nonumber\\ 
    -a_5f(\mathcal{\zeta}_1,\ldots,\mathcal{\zeta}_n)^2a_6 +f(\mathcal{\zeta}_1,\ldots,\mathcal{\zeta}_n)a_4f(\mathcal{\zeta}_1,\ldots,\mathcal{\zeta}_n)  \nonumber\\
    +f(\mathcal{\zeta}_1,\ldots,\mathcal{\zeta}_n)a_5f(\mathcal{\zeta}_1,\ldots,\mathcal{\zeta}_n)a_6
\end{align}
for all $\mathcal{\zeta}_1,\ldots,\mathcal{\zeta}_n\in \mathcal{R}$. From Fact \ref{fac2} it is clear that the ring $\mathcal{R}$ and $\mathcal{U}$ satisfy the same generalized polynomial identity. If $\mathcal{C}$ is infinite then $h(\mathcal{\zeta}_1,\ldots,\mathcal{\zeta}_n) = 0$ for all $\mathcal{\zeta}_1,\ldots,\mathcal{\zeta}_n\in \mathcal{U} \otimes_\mathcal{C} \bar{\mathcal{C}}$. Since $\mathcal{U}$ and $\mathcal{U} \otimes_\mathcal{C} \bar{\mathcal{C}}$  are prime and centrally closed \cite{erickson1975pr}, we can replace $\mathcal{R}$ with $\mathcal{U}$ or $\mathcal{U} \otimes_\mathcal{C} \bar{\mathcal{C}}$ according as $\mathcal{C}$ is finite or infinite. Then from Martindale’s theorem \cite{martindale1969pr}, ring  $\mathcal{R}$ is a primitive ring with non-zero socle $soc(\mathcal{R})$ and with $\mathcal{C}$ as its associated division ring. Now by Jacobson
density theorem \cite{N. Jacobson1956}, $\mathcal{R}$ is isomorphic to a dense ring of a linear transformation of
a vector space $\mathcal{W}$ over the division ring $\mathcal{C}$. Now following two cases arises:

\textbf{Case 1:} If the dimension of the vector space $\mathcal{W}$ is finite then from density of $\mathcal{R}$ we get ${\mathcal{R}} \cong \mathcal{M}_l(\mathcal{C}).$  Since $f(\zeta)$ is non central on $\mathcal{R}$ therefore $l$ must be greater than 1. Thus, we get the desired result by Lemma \ref{lem2}.

 \textbf{Case 2:} If the dimension of the vector space $\mathcal{W}$ is infinite. Suppose none $a_2,a_3,a_5$ and $a_6$ is central. Then by Martindale’s theorem \cite{martindale1969pr}, there exists a non-zero idempotent element $a\in \mathcal{R}$ such that $a\mathcal{R}a\cong \mathcal{M}_l(\mathcal{C}),$ where $l$ is the dimension of the vector space $\mathcal{W}a$ of  over the field  $\mathcal{C}$. Since $a_2,a_3,a_5, a_6 \notin \mathcal{C}$, there exist $b_2,b_3,b_5,b_6 \in soc(\mathcal{R})$ such that $[a_2,b_2]\neq 0,\ [a_3,b_3]\neq 0, \ [a_5,b_5]\neq 0$ and $[a_6,b_6]\neq 0$. Now by Litoff's theorem \cite{c}, there exists an idempotent $a \in soc(\mathcal{R})$ s.t. $a_2b_2,b_2a_2,a_3b_3,b_3a_3,a_5b_5,b_5a_5,a_6b_6,b_6a_6\in a\mathcal{R}a$. As we have
\begin{align*}
  a\Big( a_1f(a\zeta_1a,\ldots ,a\zeta_na)^2 +a_2f(a\zeta_1a,\ldots ,a\zeta_na)a_3f(a\zeta_1a,\ldots ,a\zeta_na)\\
   +f(a\zeta_1a,\ldots ,a\zeta_na)a_5f(a\zeta_1a,\ldots ,a\zeta_na)a_6  -a_5f(a\zeta_1a,\ldots ,a\zeta_na)^2a_6\\
   +f(a\zeta_1a,\ldots ,a\zeta_na)a_4f(a\zeta_1a,\ldots ,a\zeta_na)\Big)a  =0, \forall \ \zeta_1,\ldots ,\zeta_n\in \mathcal{R}.
\end{align*}
Thus the subring $a\mathcal{R}a$ satisfies
\begin{align*}
   aa_1af(\zeta_1,\ldots ,\zeta_n)^2 +aa_2af(\zeta_1,\ldots ,\zeta_n)aa_3af(\zeta_1,\ldots ,\zeta_n)\\
   +f(\zeta_1,\ldots ,\zeta_n)aa_5af(\zeta_1,\ldots ,\zeta_n)aa_6a-aa_5af(\zeta_1,\ldots ,\zeta_n)^2aa_6a\\
   +f(\zeta_1,\ldots ,\zeta_n)aa_4af(\zeta_1,\ldots ,\zeta_n)   =0, \forall \ \zeta_1,\ldots ,\zeta_n\in \mathcal{R}.
\end{align*}
Now by Lemma \ref{lem2} one of the following pair $aa_2a,aa_5a$ or $aa_2a,aa_6a$ or $aa_3a, aa_5a$ or $aa_3a,aa_6a$ is a central element of $a\mathcal{R}a$,  but this shows that both term of any one of the following pairs $[a_2,b_2],[a_5,b_5]$, or $[a_2,b_2],[a_6,b_6]$ or $[a_3,b_3],[a_5,b_5]$ $[a_3,b_3],[a_6,b_6]$ equals to zero, which leads to a contradiction.  
\end{proof}	
\textbf{Proof of Proposition \ref{prop1}:} From the hypothesis there exist $a,b,c,d,b^{\prime},\tilde{b}\in \mathcal{U}$ such that $\mathcal{F}(\chi)= a\chi+\tilde{b}\alpha(\chi)b$ and $\mathcal{G}(\chi)=c\chi+b^{\prime}\alpha(\chi)d$, for all $\chi \in \mathcal{R}$. Then from Fact \ref{fac2} and Fact \ref{fac3},  $\mathcal{R}$ satisfies
	
	\begin{align}\label{eq p1}
    af(\chi)^2 +\tilde{b}\alpha(f(\chi))bf(\chi)+f(\chi)cf(\chi)+f(\chi)b^{\prime}\alpha(f(\chi))d\nonumber \\ -cf(\chi)^2-b^{\prime}\alpha(f(\chi)^2)d=0.
    \end{align}
	for all $\chi=(\chi_1,\ldots,\chi_n)\in \mathcal{R}^n$. If $\alpha$ is inner automorphism, there exists $t\in \mathcal{U}$ such that $\alpha(\chi) = t\chi t^{-1}$ for all $\chi \in \mathcal{R}$, then Equation (\ref{eq p1}) reduces to:
	\begin{align}\label{eq p2}
    (a-c)f(\chi)^2 +\tilde{b}tf(\chi)t^{-1}bf(\chi)+f(\chi)cf(\chi)+f(\chi)b^{\prime}tf(\chi)t^{-1}d \nonumber \\
    -b^{\prime}tf(\chi)^2t^{-1}d=0.
    \end{align}
From Lemma \ref{lem4}, one of the following holds:
	\begin{enumerate}
	    \item $b^{\prime}t,\Tilde{b}t\in \mathcal{C}$; 
	    \item $\tilde{b}t,t^{-1}d\in \mathcal{C}$;
	    \item $t^{-1}b,b^{\prime}t\in \mathcal{C}$;
	    \item  $t^{-1}b,t^{-1}d\in \mathcal{C}.$
	\end{enumerate}
\textbf{Case 1:}	Let $b^{\prime}t,\tilde{b}t\in \mathcal{C}$. Then Equation (\ref{eq p2}) reduces to 
	\begin{align*}
    (a-c)f(\chi)^2 +f(\chi)(\tilde{b}b+c)f(\chi)=0.
    \end{align*}
for all $\chi=(\chi_1,\ldots,\chi_n)\in \mathcal{R}^n$. By using Fact \ref{fac 8}, we get
 $c-a,\ \tilde{b} b+c \in \mathcal{C}, \ \tilde{b} b+a=0$. Thus in this case we have $\mathcal{F}(\chi)= a\chi +\chi \tilde{b} b$ and $\mathcal{G}(\chi)= c\chi +\chi b^{\prime} d$,  which is the Conclusion $(1)$.\\
 \textbf{Case 2:} Let $\tilde{b}t,t^{-1}d\in \mathcal{C}$. Then Equation (\ref{eq p2}) reduces to 
	\begin{align*}
    (a-c-b^{\prime}d)f(\chi)^2 +f(\chi)(\tilde{b}b+c+b^{\prime}d)f(\chi)=0.
    \end{align*}
for all $\chi=(\chi_1,\ldots,\chi_n)\in \mathcal{R}^n$, again by Fact \ref{fac 8}, we get
 $(c+b^{\prime}d-a),\ (\tilde{b}b+c+b^{\prime}d)\in \mathcal{C}$ and $\tilde{b}b+a=0$. Thus in this case we have $\mathcal{F}(\chi)= a\chi+\chi\tilde{b}b$ and $\mathcal{G}(\chi)=c\chi+b^{\prime}d\chi$ which is the Conclusion $(2)$.\\
\textbf{Case 3:} Let  $t^{-1}b,b^{\prime}t\in \mathcal{C}$. Then Equation (\ref{eq p2}) reduces to 
	\begin{align*}
    (a-c+\tilde{b}b)f(\chi)^2 +f(\chi)cf(\chi)=0.
    \end{align*}
for all $\chi=(\chi_1,\ldots,\chi_n)\in \mathcal{R}^n$, again by Fact \ref{fac 8}, we get
$c\in \mathcal{C}$ and $\tilde{b} b+a=0$, and the functions $\mathcal{F}$ and $\mathcal{G}$ are given by $\mathcal{F}(\chi)= a \chi + \tilde{b} b\chi=(a+\tilde{b} b)\chi=0 $ and $\mathcal{G}(\chi)=c \chi +\chi b^{\prime} d$ which is the Conclusion $(3)$.
    
\textbf{Case 4:} Let  $t^{-1}b,t^{-1}d\in C.$ Then Equation (\ref{eq p2}) reduces to 
	\begin{align*}
    (a-c+\tilde{b}b-b^{\prime}d)f(\chi)^2 +f(\chi)(c+b^{\prime}d)f(\chi)=0.
    \end{align*}
for all $\chi=(\chi_1,\ldots,\chi_n)\in \mathcal{R}^n$, again by Fact \ref{fac 8}, we get $(a-c+\tilde{b}b-b^{\prime}d),(c+b^{\prime}d)\in \mathcal{C}$ and $\tilde{b}b+a=0$, and the functions $\mathcal{F}$ and $\mathcal{G}$ are given by $\mathcal{F}(\chi)= a\chi+\tilde{b}b\chi=(a+\tilde{b}b)\chi=0$ and $\mathcal{G}(\chi)=c\chi+b^{\prime}d\chi=(c+b^{\prime}d)\chi$ which is the Conclusion $(4)$.
    
Next, assume the case when the automorphism $\alpha$ is outer. Then the prime ring  $\mathcal{R}$ satisfies the following generalized polynomial identity \cite{C. L. Chuang. 1988}
\begin{align}\label{eq p3}
    af(\chi)^2 +\tilde{b}f^{\alpha}(\mathcal{Y}_1,\ldots, \mathcal{Y}_n)bf(\chi)+f(\chi)cf(\chi)    +f(\chi)b^{\prime}f^{\alpha}(\mathcal{Y}_1,\ldots, \mathcal{Y}_n)d\nonumber\\
    -cf(\chi)^2    -b^{\prime}f^{\alpha}(\mathcal{Y}_1,\ldots, \mathcal{Y}_n)^2d=0.
    \end{align}
   for all $\chi=(\chi_1,\ldots,\chi_n)\in \mathcal{R}^n$ and $\mathcal{Y}_1,\ldots, \mathcal{Y}_n\in \mathcal{R}$. In particular $\mathcal{R}$ satisfies, 
   \begin{align}\label{eq p4}
       \tilde{b}f^{\alpha}(\mathcal{Y}_1,\ldots, \mathcal{Y}_n)bf(\chi)+f(\chi)b^{\prime}f^{\alpha}(\mathcal{Y}_1,\ldots, \mathcal{Y}_n)d=0
   \end{align}
   and 
 \begin{align}\label{eq p5}
     b^{\prime}f^{\alpha}(\mathcal{Y}_1,\ldots, \mathcal{Y}_n)^2d=0
 \end{align}
 for all $\chi=(\chi_1,\ldots,\chi_n)\in \mathcal{R}^n$ and $\mathcal{Y}_1,\ldots, \mathcal{Y}_n\in \mathcal{R}$. By application of Lemma, $3$ of \cite{V. D. Filippis 2021} in Equation (\ref{eq p4}), we have  $ \tilde{b}f^{\alpha}(\mathcal{Y}_1,\ldots, \mathcal{Y}_n)b=-b^{\prime}f^{\alpha}(\mathcal{Y}_1,\ldots, \mathcal{Y}_n)d\in \mathcal{C}$. Then by application of Note $16$ of \cite{V. D. Filippis 2021}, we get that either $\tilde{b}=0$ or $b=0$ and either $b^\prime =0$ or $d=0$. In either case Equation (\ref{eq p1}) reduces to
 \begin{align}
     (a-c)f(\chi)^2+f(\chi)cf(\chi)=0
 \end{align}
for all $\chi=(\chi_1,\ldots,\chi_n)\in \mathcal{R}^n$. Therefore by Fact \ref{fac 8}, we get  $a=0$, $c \in \mathcal{C}$, which is the Conclusion $(4)$.
 \section*{Proof of the main Theorem \ref{thm}}
 From \cite{de2018}, there exist $a, c, b^{\prime}, \tilde{b}\in \mathcal{U}$ such that $\mathcal{F}(\chi) = a\chi +\tilde{b}g(\chi)$ and $\mathcal{G}(\chi) = c\chi + b^{\prime}h (\chi)$ for all $\chi \in \mathcal{R}$, where $g$ and $h$ are skew derivations associated with automorphism $\alpha$. If both $g$ and $h$ are skew inner derivations then we have the Conclusion from Proposition \ref{prop1}. Now assume that at least one of $g$ or $h$ is outer. Then from the hypothesis $\mathcal{R}$ satisfies
  \begin{align}\label{m1}
    (a-c)f(\zeta)^2 +\tilde{b}g(f(\zeta))f(\zeta)+f(\zeta)cf(\zeta)+f(\zeta)b^{\prime} h(f(\zeta))-b^{\prime}h(f(\zeta)^2)    =0.
    \end{align}
    for all $\mathcal{\zeta}=(\mathcal{\zeta}_1,\ldots,\mathcal{\zeta}_n)\in \mathcal{R}^n$.
 \begin{center}
     \textbf{$g$ and $h$ are $\mathcal{C}$-linearly independent modulo SDin}
 \end{center}
 In this case after applying the definition of $g$ and $h$, $\mathcal{R}$ satisfies the following:
  	\begin{align}
    (a-c)f(\chi_1,\ldots,\chi_n)^2 +\tilde{b}f^g(\chi_1,\ldots, \chi_n)f(\chi_1,\ldots,\chi_n)\nonumber \\
    +f(\chi_1,\ldots,\chi_n)cf(\chi_1,\ldots,\chi_n)+f(\chi_1,\ldots,\chi_n)b^{\prime} f^h(\chi_1,\ldots,\chi_n))\nonumber \\
    -b^{\prime}\alpha(f(\chi_1,\ldots,\chi_n))f^h(\chi_1,\ldots,\chi_n) -b^{\prime}f^h(\chi_1,\ldots,\chi_n)f(\chi_1,\ldots,\chi_n) \nonumber \\
    +\bigg[\tilde{b}\sum\limits_{\sigma\in \mathcal{S}_n}\alpha(\gamma_{\sigma})\sum\limits_{j=0}^{n-1}\alpha(\chi_{\sigma(1)}\chi_{\sigma(2)}\ldots, \chi_{\sigma(j)})g(\chi_{\sigma(j+1)})\chi_{\sigma(j+2)},\ldots,\chi_{\sigma(n)} \bigg]\nonumber \\ f(\chi_1,\ldots,\chi_n)+f(\chi_1,\ldots,\chi_n)\nonumber \\
    \bigg[b^{\prime}\sum\limits_{\sigma\in \mathcal{S}_n}\alpha(\gamma_{\sigma})\sum\limits_{j=0}^{n-1}\alpha(\chi_{\sigma(1)}\chi_{\sigma(2)}\ldots, \chi_{\sigma(j)})h(\chi{\sigma(j+1)})\chi_{\sigma(j+2)},\ldots,\chi_{\sigma(n)}\bigg]\nonumber \\
    -\bigg[b^{\prime}\sum\limits_{\sigma\in \mathcal{S}_n}\alpha(\gamma_{\sigma})\sum\limits_{j=0}^{n-1}\alpha(\chi_{\sigma(1)}\chi_{\sigma(2)}\ldots, \chi_{\sigma(j)})h(\chi{\sigma(j+1)})\chi_{\sigma(j+2)},\ldots,\chi_{\sigma(n)}\bigg]\nonumber\\ 
     f(\chi_1,\ldots,\chi_n)-b^{\prime}\alpha(f(\chi_1,\ldots,\chi_n))\nonumber\\ 
     \bigg[\sum\limits_{\sigma\in \mathcal{S}_n}\alpha(\gamma_{\sigma})\sum\limits_{j=0}^{n-1}\alpha(\chi_{\sigma(1)} \chi_{\sigma(2)}     \ldots, \chi_{\sigma(j)}) h(\chi_{\sigma(j+1)})\chi_{\sigma(j+2)},\ldots,\chi_{\sigma(n)}\bigg]
     =0
    \end{align}
for all $\chi_1,\ldots,\chi_n\in \mathcal{R}$. From \cite{C. L. Chuang. 1988}, $\mathcal{R}$ satisfies
 	\begin{align}\label{m2}
    (a-c)f(\chi_1,\ldots,\chi_n)^2 +\tilde{b}f^g(\chi_1,\ldots, \chi_n)f(\chi_1,\ldots,\chi_n)\nonumber \\
    +f(\chi_1,\ldots,\chi_n)cf(\chi_1,\ldots,\chi_n)+f(\chi_1,\ldots,\chi_n)b^{\prime} f^h(\chi_1,\ldots,\chi_n))\nonumber \\
    -b^{\prime}\alpha(f(\chi_1,\ldots,\chi_n))f^h(\chi_1,\ldots,\chi_n)  -b^{\prime}f^h(\chi_1,\ldots,\chi_n)f(\chi_1,\ldots,\chi_n)\nonumber \\
    +\bigg[\tilde{b}\sum\limits_{\sigma\in \mathcal{S}_n}\alpha(\gamma_{\sigma})\sum\limits_{j=0}^{n-1}\alpha(\chi_{\sigma(1)}\chi_{\sigma(2)}\ldots, \chi_{\sigma(j)})\mathcal{U}_{\sigma(j+1)}\chi_{\sigma(j+2)},\ldots,\chi_{\sigma(n)}\bigg]\nonumber \\ f(\chi_1,\ldots,\chi_n)+f(\chi_1,\ldots,\chi_n)\nonumber \\
    \bigg[b^{\prime}\sum\limits_{\sigma\in \mathcal{S}_n}\alpha(\gamma_{\sigma})\sum\limits_{j=0}^{n-1}\alpha(\chi_{\sigma(1)}\chi_{\sigma(2)}\ldots, \chi_{\sigma(j)}) \mathcal{V}_{\sigma(j+1)}\chi_{\sigma(j+2)},\ldots,\chi_{\sigma(n)}\bigg]\nonumber \\
    -\bigg[b^{\prime}\sum\limits_{\sigma\in \mathcal{S}_n}\alpha(\gamma_{\sigma})\sum\limits_{j=0}^{n-1}\alpha(\chi_{\sigma(1)}\chi_{\sigma(2)}\ldots, \chi_{\sigma(j)})\mathcal{V}_{\sigma(j+1)}\chi_{\sigma(j+2)},\ldots,\chi_{\sigma(n)}\bigg]\nonumber\\
     f(\chi_1,\ldots,\chi_n)-b^{\prime}\alpha(f(\chi_1,\ldots,\chi_n))\nonumber \\ \bigg[\sum\limits_{\sigma\in \mathcal{S}_n}\alpha(\gamma_{\sigma})\sum\limits_{j=0}^{n-1}\alpha(\chi_{\sigma(1)}\chi_{\sigma(2)}\ldots, \chi_{\sigma(j)}) \mathcal{V}_{\sigma(j+1)}\chi_{\sigma(j+2)},\ldots,\chi_{\sigma(n)}\bigg]\nonumber \\
    =0
    \end{align}
 for all $\chi_1,\ldots,\chi_n,\mathcal{U}_1,\ldots, \mathcal{U}_n,\mathcal{V}_1,\ldots,\mathcal{V}_n\in \mathcal{R}$. In particular, $\mathcal{R}$ satisfies:
 \begin{align}\label{m3}
   \tilde{b} \bigg[\sum\limits_{\sigma\in \mathcal{S}_n}\alpha(\gamma_{\sigma})\sum\limits_{j=0}^{n-1}\alpha(\chi_{\sigma(1)}\chi_{\sigma(2)}\ldots, \chi_{\sigma(j)})\mathcal{U}_{\sigma(j+1)}\chi_{\sigma(j+2)},\ldots,\chi_{\sigma(n)}\bigg]\nonumber\\
    f(\chi_1,\ldots,\chi_n)     =0
 \end{align}
for all $\chi_1,\ldots,\chi_n,\mathcal{U}_1,\ldots, \mathcal{U}_n\in \mathcal{R}$. If the automorphism $\alpha$ is not inner then from \cite{C. L. Chuang. 1988}, the ring $\mathcal{R}$ satisfies
 \begin{align}\label{m4}
   \tilde{b} \bigg[\sum\limits_{\sigma\in \mathcal{S}_n}\alpha(\gamma_{\sigma})\sum\limits_{j=0}^{n-1}\mathcal{Y}_{\sigma(1)}\mathcal{Y}_{\sigma(2)}\ldots, \mathcal{Y}_{\sigma(j)},\mathcal{U}_{\sigma(j+1)}\chi_{\sigma(j+2)},\ldots,\chi_{\sigma(n)}\bigg] f(\chi_1,\ldots,\chi_n)  \nonumber\\
    =0
 \end{align}
 for all $\chi_1,\ldots,\chi_n,\mathcal{Y}_1,\ldots,\mathcal{Y}_n,\mathcal{U}_1,\ldots, \mathcal{U}_n\in \mathcal{R}$. In particular, choosing $\mathcal{U}_1=\cdots =\mathcal{U}_{n-1}=0$, we have 
 \begin{align}\label{m5}
    \tilde{b}f^{\alpha}(\mathcal{Y}_1,\ldots,\mathcal{Y}_{n-1},\mathcal{U}_n) f(\chi_1,\ldots,\chi_n)  =0
 \end{align}
 for all $\chi_1,\ldots,\chi_n, \mathcal{Y}_1,\ldots, \mathcal{Y}_{n-1},\mathcal{U}_n\in \mathcal{R}$. Since $f(\chi_1,\ldots,\chi_n)$ is non central polynomial therefore $\tilde{b}=0.$
 
Also, from Equation (\ref{m2}), we have
\begin{align*}
    f(\chi_1,\ldots,\chi_n) \bigg[b^{\prime}\sum\limits_{\sigma\in \mathcal{S}_n}\alpha(\gamma_{\sigma})\sum\limits_{j=0}^{n-1}\mathcal{Y}_{\sigma(1)}\mathcal{Y}_{\sigma(2)}\ldots, \mathcal{Y}_{\sigma(j)} \mathcal{V}_{\sigma(j+1)}\chi_{\sigma(j+2)},\ldots,\chi_{\sigma(n)}\bigg]\nonumber \\
     -b^{\prime}f^{\alpha}(\mathcal{Y}_1,\ldots,\mathcal{Y}_n)\bigg[\sum\limits_{\sigma\in \mathcal{S}_n}\alpha(\gamma_{\sigma})\sum\limits_{j=0}^{n-1}\mathcal{Y}_{\sigma(1)}\mathcal{Y}_{\sigma(2)}\ldots,\mathcal{Y}_{\sigma(j)}  \mathcal{V}_{\sigma(j+1)}\chi_{\sigma(j+2)},\ldots,\chi_{\sigma(n)}\bigg]  \nonumber\\ 
    -\bigg[b^{\prime}\sum\limits_{\sigma\in\mathcal{S}_n}\alpha(\gamma_{\sigma})\sum\limits_{j=0}^{n-1}\mathcal{Y}_{\sigma(1)}\mathcal{Y}_{\sigma(2)}\ldots, \mathcal{Y}_{\sigma(j)}\mathcal{V}_{\sigma(j+1)}\chi_{\sigma(j+2)},\ldots,\chi_{\sigma(n)}\bigg]f(\chi_1,\ldots,\chi_n) =0
\end{align*}
for all  $\chi_1,\ldots,\chi_n,\mathcal{Y}_1,\ldots, \mathcal{Y}_{n},\mathcal{V}_1,\ldots,\mathcal{V}_n\in \mathcal{R}$. In particular, choosing $\mathcal{V}_1=\cdots =\mathcal{V}_{n-1}=\mathcal{Y}_n=0$, we have 
\begin{align} \label{ed1}
    f(\chi_1,\ldots,\chi_n) b^{\prime}f(\mathcal{Y}_1,\ldots,\mathcal{Y}_{n-1},\mathcal{V}_n)- b^{\prime}f(\mathcal{Y}_1,\ldots,\mathcal{Y}_{n-1},\mathcal{V}_n)f(\chi_1,\ldots,\chi_n)=0
\end{align}
for all $\chi_1,\dots, \chi_n,\mathcal{Y}_1,\ldots,\mathcal{Y}_{n-1},\mathcal{V}_n\in \mathcal{R}$.  By application of Lemma, $3$ of \cite{V. D. Filippis 2021} in Equation (\ref{ed1}), we have  $ b^{\prime}f(\mathcal{Y}_1,\ldots,\mathcal{Y}_{n-1},\mathcal{V}_n)\in \mathcal{C}$. Then by application of Note $16$ of \cite{V. D. Filippis 2021}, we get that  $b^{\prime}=0$, and earlier we have shown that $\tilde{b}$=0, this leads to a contradiction.

Now if the automorphism $\alpha$ is inner, then there exists $t\in \mathcal{R}$ such that $\alpha(\chi) = t\chi t^{-1}$ for all $\chi \in \mathcal{R}$. Now from Equation (\ref{m3}), $\mathcal{R}$ satisfies the following equation
 \begin{multline}\label{m6}
    \bigg[\tilde{b}t\sum\limits_{\sigma\in \mathcal{S}_n}\gamma_{\sigma}\sum\limits_{j=0}^{n-1}\chi_{\sigma(1)}\chi_{\sigma(2)}\ldots, \chi_{\sigma(j)}t^{-1}\mathcal{U}_{\sigma(j+1)}\chi_{\sigma(j+2)},\ldots,\chi_{\sigma(n)}\bigg] f(\chi_1,\ldots,\chi_n)  \\
    =0
 \end{multline}
 for all  $\chi_1,\ldots,\chi_n,\mathcal{U}_1 \ldots, \mathcal{U}_{n}\in \mathcal{R}$ . In particular, choosing $\mathcal{U}_1 =\ldots, =\mathcal{U}_{n-1} = 0$ and $\mathcal{U}_n =t\chi_n$, we have that
 \begin{align}\label{m7}
    \tilde{b}t f(\chi_1,\ldots,\chi_n)^2=0
 \end{align}
 for all $\chi_1,\ldots,\chi_n\in \mathcal{R}$, and this gives that $\tilde{b}=0$.
Also from the Equation (\ref{m2}) we have
\begin{multline*}
    f(\chi_1,\ldots,\chi_n) \bigg[b^{\prime}t\sum\limits_{\sigma\in \mathcal{S}_n}\gamma_{\sigma}\sum\limits_{j=0}^{n-1}{\chi_{\sigma(1)}\chi_{\sigma(2)}\ldots, \chi_{\sigma(j)} t^{-1}\mathcal{V}_{\sigma(j+1)}\chi_{\sigma(j+2)},\ldots,\chi_{\sigma(n)}}\bigg]\nonumber \\
     -b^{\prime}tf(\chi_1,\ldots,\chi_n))\bigg[\sum\limits_{\sigma\in \mathcal{S}_n}\gamma_{\sigma}\sum\limits_{j=0}^{n-1}{\chi_{\sigma(1)}\chi_{\sigma(2)}\ldots, \chi_{\sigma(j)} t^{-1}\mathcal{V}_{\sigma(j+1)}\chi_{\sigma(j+2)},\ldots,\chi_{\sigma(n)}}\bigg]\nonumber \\
      -b^{\prime}t\bigg[\sum\limits_{\sigma\in \mathcal{S}_n}\gamma_{\sigma}\sum\limits_{j=0}^{n-1}{\chi_{\sigma(1)}\chi_{\sigma(2)}\ldots, \chi_{\sigma(j)} t^{-1}\mathcal{V}_{\sigma(j+1)}\chi_{\sigma(j+2)},\ldots,\chi_{\sigma(n)}}\bigg]f(\chi_1,\ldots,\chi_n)=0
\end{multline*}
 for all $\chi_1,\ldots,\chi_n,\mathcal{V}_1,\ldots,\mathcal{V}_n\in \mathcal{R}$. In particular, choosing $\mathcal{V}_1 \ldots,  \mathcal{V}_{n-1} = 0$ and $\mathcal{V}_n = t\chi_n$, we have that
\begin{align}\label{ed2}
  f(\chi_1,\ldots,\chi_n)b^{\prime}tf(\chi_1,\ldots,\chi_n)-2  b^{\prime}t f(\chi_1,\ldots,\chi_{n})^2=0
 \end{align}
 for all $\chi_1,\ldots, \chi_{n}\in \mathcal{R}$. By Fact \ref{fac 8}, we have $b^{\prime}t=0$ and since $t$ is invertible, $b^{\prime}=0$, which  again leads to a contradiction.
  \begin{center}
     \textbf{$g$ and $h$ are $\mathcal{C}$-linearly dependent modulo SDin}
 \end{center}
 Let the derivations $g$ and $h$ are $\mathcal{C}$-linearly dependent. Then there exist $p,q\in \mathcal{C}$ and $\Gamma\in Aut(\mathcal{R})$ such that $pg(\chi) +qh(\chi) = w\chi- \Gamma(\chi)w$ for all $\chi\in \mathcal{R}$ and some non-zero $w\in \mathcal{R}$. We divide the proof into the following cases:
 
\textbf{Case 1:} Suppose $p=0$ and $q\neq 0$. In this case $h(\chi) =q^{-1} w\chi -q^{-1} \Gamma(\chi)w$ by Fact (\ref{fac 9}), $\Gamma=\alpha$. For $\lambda=q^{-1}w$, we can write $h(\chi) =\lambda \chi - \alpha(\chi)\lambda$ and so $\mathcal{F}(\chi)=a\chi+\tilde{b}g(\chi)$ and $\mathcal{G}(\chi)=(c+b^{\prime}\lambda)\chi-b^{\prime}\alpha(\chi)\lambda$. Then from Equation (\ref{m1}), we
get
	\begin{align}\label{m8}
    (a-c-b^{\prime}\lambda)f(\chi_1,\ldots,\chi_n)^2 +\tilde{b}f^g(\chi_1,\ldots, \chi_n)f(\chi_1,\ldots,\chi_n)\nonumber \\
    +f(\chi_1,\ldots,\chi_n)(c+b^{\prime}\lambda)f(\chi_1,\ldots,\chi_n)-f(\chi_1,\ldots,\chi_n)b^{\prime} \alpha(f(\chi_1,\ldots,\chi_n))\lambda\nonumber \\
    -b^{\prime}\alpha(f(\chi_1,\ldots,\chi_n)^2)\lambda
    +\tilde{b}\bigg[\sum\limits_{\sigma\in\mathcal{S}_n}\alpha(\gamma_{\sigma})\sum\limits_{j=0}^{n-1}\alpha(\chi_{\sigma(1)}\chi_{\sigma(2)}\ldots,\chi_{\sigma(j)})\nonumber \\ \mathcal{U}_{\sigma(j+1)}\chi_{\sigma(j+2)},\ldots,\chi_{\sigma(n)}\bigg] f(\chi_1,\ldots,\chi_n)=0
    \end{align}
 for all $\chi_1,\ldots,\chi_n, \mathcal{U}_1,\ldots, \mathcal{U}_n\in \mathcal{R}$. In particular, we have
 	\begin{align}\label{m9}
    \tilde{b}\bigg[\sum\limits_{\sigma\in \mathcal{S}_n}\alpha(\gamma_{\sigma})\sum\limits_{j=0}^{n-1}\alpha(\chi_{\sigma(1)}\chi_{\sigma(2)}\ldots, \chi_{\sigma(j)})\mathcal{U}_{\sigma(j+1)}\chi_{\sigma(j+2)},\ldots,\chi_{\sigma(n)}\bigg] f(\chi_1,\ldots,\chi_n)\nonumber\\
    =0
    \end{align}
  for all $\chi_1,\ldots,\chi_n, \mathcal{U}_1,\ldots, \mathcal{U}_n\in \mathcal{R}$. Equation (\ref{m9}) is similar to Equation (\ref{m3}) thus by parallel argument used previously, we get that $\tilde{b}=0$, but this makes both $\mathcal{F}$ and $\mathcal{G}$ generalized  inner skew derivations which contradict to our assumption.
 
\textbf{Case (2):} Suppose $p\neq 0$ and $q= 0$. In this case $g(\chi) = p^{-1}w\chi-p^{-1}\Gamma(\chi)w$ and by Fact (\ref{fac 9}), $\Gamma=\alpha$. Suppose $p^{-1}w=\lambda^{\prime}$ then $g(\chi) = \lambda^{\prime}\chi-\alpha(\mathcal{X})\lambda^{\prime}$. We write $\mathcal{F}(\chi)=a\chi+\tilde{b}(\lambda^{\prime}\chi-\alpha(\chi)\lambda^{\prime})$. Then Equation (\ref{m1}) reduces to
	\begin{multline}\label{m12}
    (a-c+\tilde{b}\lambda^{\prime})f(\zeta)^2 +f(\zeta)cf(\zeta)
    +f(\zeta)b^{\prime} h(f(\zeta)) -b^{\prime}h(f(\zeta)^2)-\tilde{b}\alpha(f(\zeta))\lambda^{\prime}f(\zeta)=0.
    \end{multline}
 Which further can be written as
 	\begin{align}\label{m13}
    (a-c+\tilde{b}\lambda^{\prime})f(\chi_1,\ldots,\chi_n)^2 +f(\chi_1,\ldots,\chi_n)cf(\chi_1,\ldots,\chi_n)\nonumber\\
    +f(\chi_1,\ldots,\chi_n)b^{\prime} f^h(\chi_1,\ldots,\chi_n) -b^{\prime}\alpha(f(\chi_1,\ldots,\chi_n))f^h(\chi_1,\ldots,\chi_n)\nonumber\\
    -\tilde{b}\alpha(f(\chi_1,\ldots,\chi_n))\lambda^{\prime}f(\chi_1,\ldots,\chi_n)+f(\chi_1,\ldots,\chi_n)\nonumber\\
    \bigg[b^{\prime}\sum\limits_{\sigma\in \mathcal{S}_n}\alpha(\gamma_{\sigma})\sum\limits_{j=0}^{n-1}\alpha(\chi_{\sigma(1)}\chi_{\sigma(2)}\ldots, \chi_{\sigma(j)}) \mathcal{V}_{\sigma(j+1)}\chi_{\sigma(j+2)},\ldots,\chi_{\sigma(n)}\bigg]\nonumber\\
    -b^{\prime}\bigg[\sum\limits_{\sigma\in \mathcal{S}_n}\alpha(\gamma_{\sigma})\sum\limits_{j=0}^{n-1}\alpha(\chi_{\sigma(1)}\chi_{\sigma(2)}\ldots, \chi_{\sigma(j)}) \mathcal{V}_{\sigma(j+1)}\chi_{\sigma(j+2)},\ldots,\chi_{\sigma(n)}\bigg]\nonumber\\
    f(\chi_1,\ldots,\chi_n)    -b^{\prime}\alpha(f(\chi_1,\ldots,\chi_n)) \nonumber \\
    \bigg[\sum\limits_{\sigma\in \mathcal{S}_n}\alpha(\gamma_{\sigma})\sum\limits_{j=0}^{n-1}\alpha(\chi_{\sigma(1)}\chi_{\sigma(2)}\ldots, \chi_{\sigma(j)})\mathcal{V}_{\sigma(j+1)}\chi_{\sigma(j+2)},\ldots,\chi_{\sigma(n)}\bigg]\nonumber\\
   -b^{\prime}f^h(\chi_1,\ldots,\chi_n)f(\chi_1,\ldots,\chi_n)=0
    \end{align}
    for all $\chi_1,\ldots,\chi_n, \mathcal{V}_1,\ldots, \mathcal{V}_n\in \mathcal{R}$. In particular, $\mathcal{R}$ satisfies
    \begin{align}
        f(\chi_1,\ldots,\chi_n) \nonumber\\
        b^{\prime}\bigg[\sum\limits_{\sigma\in \mathcal{S}_n}\alpha(\gamma_{\sigma})\sum\limits_{j=0}^{n-1}\alpha(\chi_{\sigma(1)}\chi_{\sigma(2)}\ldots, \chi_{\sigma(j)}) \mathcal{V}_{\sigma(j+1)}\chi_{\sigma(j+2)},\ldots,\chi_{\sigma(n)}\bigg]\nonumber\\
    -b^{\prime}\alpha(f(\chi_1,\ldots,\chi_n))  \nonumber\\ 
    \bigg[\sum\limits_{\sigma\in\mathcal{S}_n}\alpha(\gamma_{\sigma})\sum\limits_{j=0}^{n-1}\alpha(\chi_{\sigma(1)}\chi_{\sigma(2)}\ldots,\chi_{\sigma(j)})\mathcal{V}_{\sigma(j+1)}\chi_{\sigma(j+2)},\ldots,\chi_{\sigma(n)}\bigg]\nonumber\\
    -b^{\prime}\bigg[\sum\limits_{\sigma\in \mathcal{S}_n}\alpha(\gamma_{\sigma})\sum\limits_{j=0}^{n-1}\alpha(\chi_{\sigma(1)}\chi_{\sigma(2)}\ldots, \chi_{\sigma(j)}) \mathcal{V}_{\sigma(j+1)}\chi_{\sigma(j+2)},\ldots,\chi_{\sigma(n)}\bigg]\nonumber\\ f(\chi_1,\ldots,\chi_n)=0
    \end{align}
    for all $\chi_1,\ldots,\chi_n,\mathcal{V}_1,\ldots,\mathcal{V}_n\in \mathcal{R}$. If $\alpha$ is not inner then $\mathcal{R}$ satisfies
     \begin{align}
        f(\chi_1,\ldots,\chi_n)\bigg[b^{\prime}\sum\limits_{\sigma\in \mathcal{S}_n}\alpha(\gamma_{\sigma})\sum\limits_{j=0}^{n-1}\mathcal{Y}_{\sigma(1)}\mathcal{Y}_{\sigma(2)}\ldots, \mathcal{Y}_{\sigma(j)} \mathcal{V}_{\sigma(j+1)}\chi_{\sigma(j+2)},\ldots,\chi_{\sigma(n)}\bigg]\nonumber\\
    -b^{\prime}f^{\alpha}(\mathcal{Y}_1,\ldots,\mathcal{Y}_n) \bigg[\sum\limits_{\sigma\in \mathcal{S}_n}\alpha(\gamma_{\sigma})\sum\limits_{j=0}^{n-1}\mathcal{Y}_{\sigma(1)}\mathcal{Y}_{\sigma(2)}\ldots, \mathcal{Y}_{\sigma(j)} \mathcal{V}_{\sigma(j+1)}\chi_{\sigma(j+2)},\ldots,\chi_{\sigma(n)}\bigg]   \nonumber \\
   -b^{\prime}\bigg[\sum\limits_{\sigma\in \mathcal{S}_n}\alpha(\gamma_{\sigma})\sum\limits_{j=0}^{n-1}\mathcal{Y}_{\sigma(1)}\mathcal{Y}_{\sigma(2)}\ldots, \mathcal{Y}_{\sigma(j)} \mathcal{V}_{\sigma(j+1)}\chi_{\sigma(j+2)},\ldots,\chi_{\sigma(n)}\bigg] f(\chi_1,\ldots,\chi_n) \nonumber \\
   =0
    \end{align}
   for all  $\chi_1,\ldots,\chi_n,\mathcal{Y}_1,\ldots,\mathcal{Y}_{n},\mathcal{V}_1,\ldots,\mathcal{V}_n\in \mathcal{R}$.  In particular, choosing $\mathcal{V}_1=\cdots =\mathcal{V}_{n-1}=\mathcal{Y}_n=0$ we get
    \begin{align} \label{m14}
        f(\chi_1,\ldots,\chi_n)b^{\prime}f(\mathcal{Y}_1,\ldots,\ldots,\mathcal{Y}_{n-1},\mathcal{V}_n)  - b^{\prime}f(\chi_1,\ldots,\chi_n)f(\mathcal{Y}_1,\ldots,\ldots,\mathcal{Y}_{n-1},\mathcal{V}_n)  =0
    \end{align}
   for all  $\mathcal{Y}_1,\ldots,\mathcal{Y}_{j},\mathcal{Y}_{n-1},\mathcal{V}_n\in \mathcal{R}$. Now Equation (\ref{m14}) is similar to Equation (\ref{ed1}) by using parallel arguments  $b^\prime =0$, which leads to a contradiction.
    
    Further, if $\alpha$ is inner then there exists $t\in \mathcal{R}$ such that $\alpha(\chi)=t\chi t^{-1}\chi t$ and the ring $\mathcal{R}$ satisfies
     \begin{align}
        f(\chi_1,\ldots,\chi_n)b^{\prime}t \bigg[\sum\limits_{\sigma\in \mathcal{S}_n}\gamma_{\sigma}\sum\limits_{j=0}^{n-1}\chi_{\sigma(1)}\chi_{\sigma(2)}\ldots, \chi_{\sigma(j)}t^{-1} \mathcal{V}_{\sigma(j+1)}\chi_{\sigma(j+2)},\ldots,\chi_{\sigma(n)}\bigg] \nonumber\\
        -b^{\prime}t f(\chi_1,\ldots,\chi_n)\bigg[\sum\limits_{\sigma\in \mathcal{S}_n}\gamma_{\sigma}\sum\limits_{j=0}^{n-1}\chi_{\sigma(1)}\chi_{\sigma(2)}\ldots, \chi_{\sigma(j)}t^{-1} \mathcal{V}_{\sigma(j+1)}\chi_{\sigma(j+2)},\ldots,\chi_{\sigma(n)}\bigg] \nonumber \\
    -b^{\prime}t\bigg[\sum\limits_{\sigma\in \mathcal{S}_n}\gamma_{\sigma}\sum\limits_{j=0}^{n-1}\chi_{\sigma(1)}\chi_{\sigma(2}\ldots, \chi_{\sigma(j)}t^{-1} \mathcal{V}_{\sigma(j+1)}\chi_{\sigma(j+2)},\ldots,\chi_{\sigma(n)}\bigg]f(\chi_1,\ldots,\chi_n)\nonumber \\
    =0
    \end{align}
for all  $\chi_1,\ldots,\chi_n,\mathcal{V}_1,\ldots,\mathcal{V}_n\in \mathcal{R}$. In particular, choosing $\mathcal{V}_1=\mathcal{V}_2=\cdots=\mathcal{V}_{n-1}=0$ and $\mathcal{V}_n=t\chi_n$,  we have
    \begin{align}\label{ed4}
      f(\chi_1,\ldots,\chi_{n-1},\chi_n)b^{\prime}tf(\chi_1,\ldots,\chi_{n-1},\chi_n)-2b^{\prime}tf(\chi_1,\ldots,\chi_{n-1},\chi_n)^2=0
    \end{align}
    for all  $\chi_1,\ldots\chi_{n-1},\mathcal{Y}_n\in \mathcal{R}$. Above Equation (\ref{ed4}) is similar to Equation (\ref{ed2}) and this gives $b^\prime =0$, which again gives contradiction.
  
  \textbf{Case $(3)$:}  Suppose $p\neq  0$ and $q\neq 0$. Then for $\eta=p^{-1}w$, $\delta=-p^{-1}q,\ g(\chi)$ $=\delta h(\chi)+\eta \chi-\alpha(\chi)\eta$ for all $\chi\in \mathcal{R}$ and from Equation (\ref{m1}), the ring $\mathcal{R}$ satisfies
\begin{align}
    (a-c+\tilde{b}\eta)f(\zeta)^2+\tilde{b}(\delta h(f(\zeta))-\alpha(f(\zeta))\eta )f(\zeta)
    +f(\zeta)cf(\zeta)      +f(\zeta)b^{\prime}h(f(\zeta))   \nonumber\\   
    -b^{\prime}h(f(\zeta)^2) =0
\end{align}
    for all $\mathcal{\zeta}=(\mathcal{\zeta}_1,\ldots,\mathcal{\zeta}_n)\in \mathcal{R}^n$. Now applying the definition of $h$ we have
    \begin{align}\label{ed7}
        (a-c+\tilde{b}\eta)f(\chi_1,\ldots,\chi_n)^2+\tilde{b}\delta f^h(\chi_1,\ldots,\chi_n)f(\chi_1,\ldots,\chi_n)\nonumber\\
        -\tilde{b}\alpha(f(\chi_1,\ldots,\chi_n)) \eta f(\chi_1,\ldots,\chi_n)    +f(\chi_1,\ldots,\chi_n)cf(\chi_1,\ldots,\chi_n)\nonumber\\
    +f(\chi_1,\ldots,\chi_n)b^{\prime}f^h(\chi_1,\ldots,\chi_n)\nonumber \\-b^{\prime}\alpha(f(\chi_1,\ldots,\chi_n))f^h(\chi_1,\ldots,\chi_n)
    -b^{\prime}f^h(\chi_1,\ldots,\chi_n)f(\chi_1,\ldots,\chi_n)\nonumber \\
    +  \bigg[\tilde{b}\delta\sum\limits_{\sigma\in \mathcal{S}_n}\alpha(\gamma_{\sigma})\sum\limits_{j=0}^{n-1}\alpha(\chi_{\sigma(1)}\chi_{\sigma(2)}\ldots, \chi_{\sigma(j)}) \mathcal{V}_{\sigma(j+1)}\chi_{\sigma(j+2)},\ldots,\chi_{\sigma(n)}\bigg] f(\chi_1,\dots,\chi_n)\nonumber\\
   +f(\chi_1,\dots,\chi_n)    \bigg[b^{\prime}\sum\limits_{\sigma\in \mathcal{S}_n}\alpha(\gamma_{\sigma})\sum\limits_{j=0}^{n-1}\alpha(\chi_{\sigma(1)}\chi_{\sigma(2)}\ldots, \chi_{\sigma(j)}) \mathcal{V}_{\sigma(j+1)}\chi_{\sigma(j+2)},\ldots,\chi_{\sigma(n)}\bigg]\nonumber\\ 
    -b^{\prime}\alpha(f(\chi_1,\dots,\chi_n))     \bigg[\sum\limits_{\sigma\in \mathcal{S}_n}\alpha(\gamma_{\sigma})\sum\limits_{j=0}^{n-1}\alpha(\chi_{\sigma(1)}\chi_{\sigma(2)}\ldots, \chi_{\sigma(j)}) \mathcal{V}_{\sigma(j+1)}\chi_{\sigma(j+2)},\ldots,\chi_{\sigma(n)}\bigg]    \nonumber \\    
    -\bigg[b^{\prime}\sum\limits_{\sigma\in \mathcal{S}_n}\alpha(\gamma_{\sigma})\sum\limits_{j=0}^{n-1}\alpha(\chi_{\sigma(1)}
    \chi_{\sigma(2)}\ldots, \chi_{\sigma(j)}) \mathcal{V}_{\sigma(j+1)}\chi_{\sigma(j+2)},\ldots,\chi_{\sigma(n)}\bigg] f(\chi_1,\ldots,\chi_n)   \nonumber\\ =0
    \end{align}
    for all $\chi_1,\ldots,\chi_n,\mathcal{V}_1,\ldots,\mathcal{V}_n\in \mathcal{R}$. In particular, $\mathcal{R}$ satisfies 
    \begin{align}\label{m20}
        \bigg[\tilde{b}\delta\sum\limits_{\sigma\in \mathcal{S}_n}\alpha(\gamma_{\sigma})\sum\limits_{j=0}^{n-1}\alpha(\chi_{\sigma(1)}\chi_{\sigma(2)}\ldots, \chi_{\sigma(j)}) \mathcal{V}_{\sigma(j+1)}\chi_{\sigma(j+2)},\ldots,\chi_{\sigma(n)}\bigg]\nonumber\\ 
        f(\chi_1,\dots,\chi_n)+f(\chi_1,\dots,\chi_n)  \nonumber\\   
        \bigg[b^{\prime}\sum\limits_{\sigma\in \mathcal{S}_n}\alpha(\gamma_{\sigma})\sum\limits_{j=0}^{n-1}\alpha(\chi_{\sigma(1)}\chi_{\sigma(2)}\ldots, \chi_{\sigma(j)}) \mathcal{V}_{\sigma(j+1)}\chi_{\sigma(j+2)},\ldots,\chi_{\sigma(n)}\bigg]\nonumber\\ 
    -b^{\prime}\alpha(f(\chi_1,\dots,\chi_n))\nonumber\\ 
    \bigg[\sum\limits_{\sigma\in \mathcal{S}_n}\alpha(\gamma_{\sigma})\sum\limits_{j=0}^{n-1}\alpha(\chi_{\sigma(1)}\chi_{\sigma(2)}\ldots, \chi_{\sigma(j)}) \mathcal{V}_{\sigma(j+1)}\chi_{\sigma(j+2)},\ldots,\chi_{\sigma(n)}\bigg] \nonumber \\
    -\bigg[b^{\prime}\sum\limits_{\sigma\in \mathcal{S}_n}\alpha(\gamma_{\sigma})\sum\limits_{j=0}^{n-1}\alpha(\chi_{\sigma(1)}\chi_{\sigma(2)}\ldots, \chi_{\sigma(j)}) \mathcal{V}_{\sigma(j+1)}\chi_{\sigma(j+2)},\ldots,\chi_{\sigma(n)}\bigg]\nonumber\\ f(\chi_1,\ldots,\chi_n)
    =0
    \end{align}
      for all $\chi_1,\ldots,\chi_n,\mathcal{V}_1,\ldots,\mathcal{V}_n\in \mathcal{R}$. If $\alpha$ is not inner then by \cite{C. L. Chuang. 1988}  $\mathcal{R}$ satisfies  the following equation
     \begin{align}
        \bigg[\tilde{b}\delta\sum\limits_{\sigma\in\mathcal{S}_n}\alpha(\gamma_{\sigma})\sum\limits_{j=0}^{n-1}\mathcal{Y}_{\sigma(1)},\ldots,\mathcal{Y}_{\sigma(j)}\mathcal{V}_{\sigma(j+1)}\chi_{\sigma(j+2)},\ldots,\chi_{\sigma(n)}\bigg]           f(\chi_1,\dots,\chi_n) \nonumber\\   +f(\chi_1,\dots,\chi_n)     
           \bigg[b^{\prime}\sum\limits_{\sigma\in\mathcal{S}_n}\alpha(\gamma_{\sigma})\sum\limits_{j=0}^{n-1}\mathcal{Y}_{\sigma(1)},\ldots,\mathcal{Y}_{\sigma(j)}\mathcal{V}_{\sigma(j+1)}\chi_{\sigma(j+2)},\ldots,\chi_{\sigma(n)}\bigg]\nonumber\\ 
    -b^{\prime}f^\alpha (\mathcal{Y}_1,\dots,\mathcal{Y}_n))    \bigg[\sum\limits_{\sigma\in \mathcal{S}_n}\alpha(\gamma_{\sigma})\sum\limits_{j=0}^{n-1}\mathcal{Y}_{\sigma(1)},\ldots,\mathcal{Y}_{\sigma(j)} \mathcal{V}_{\sigma(j+1)}\chi_{\sigma(j+2)},\ldots,\chi_{\sigma(n)}\bigg] \nonumber \\
    -\bigg[b^{\prime}\sum\limits_{\sigma\in \mathcal{S}_n}\alpha(\gamma_{\sigma})\sum\limits_{j=0}^{n-1}\mathcal{Y}_{\sigma(1)},\ldots,\mathcal{Y}_{\sigma(j)} \mathcal{V}_{\sigma(j+1)}\chi_{\sigma(j+2)},\ldots,\chi_{\sigma(n)}\bigg] f(\chi_1,\ldots,\chi_n) \nonumber\\    =0
    \end{align}
    for all $\chi_1,\ldots,\chi_n,,\mathcal{Y}_1,\ldots,\mathcal{Y}_n,\mathcal{V}_1,\ldots,\mathcal{V}_n\in \mathcal{R}$. In particular, choosing $\chi_n=0$, we have that 
     \begin{align}
        \bigg[\tilde{b}\delta\sum\limits_{\sigma\in \mathcal{S}_n}\alpha(\gamma_{\sigma})\sum\limits_{j=0}^{n-1}\mathcal{Y}_{\sigma(1)},\ldots,\mathcal{Y}_{\sigma(j)} \mathcal{V}_{\sigma(j+1)}\chi_{\sigma(j+2)},\ldots,\chi_{\sigma(n)}\bigg]\ \nonumber \\
        f(\chi_1,\dots,\chi_n)   +f(\chi_1,\dots,\chi_n)   \nonumber\\   
                \bigg[b^{\prime}\sum\limits_{\sigma\in \mathcal{S}_n}\alpha(\gamma_{\sigma})\sum\limits_{j=0}^{n-1}\mathcal{Y}_{\sigma(1)}\ldots, \mathcal{Y}_{\sigma(j)}\mathcal{V}_{\sigma(j+1)}\chi_{\sigma(j+2)},\ldots,\chi_{\sigma(n)}\bigg]     \nonumber \\
    -\bigg[b^{\prime}\sum\limits_{\sigma\in \mathcal{S}_n}\alpha(\gamma_{\sigma})\sum\limits_{j=0}^{n-1}\mathcal{Y}_{\sigma(1)},\ldots,\mathcal{Y}_{\sigma(j)} \mathcal{V}_{\sigma(j+1)}\chi_{\sigma(j+2)},\ldots,\chi_{\sigma(n)}\bigg] \nonumber\\ f(\chi_1,\ldots,\chi_n)      =0
        \end{align}
       for all $\chi_1,\ldots,\chi_n,\mathcal{Y}_1,\ldots,\mathcal{Y}_n,\mathcal{V}_1,\ldots,\mathcal{V}_n\in \mathcal{R}$. Again choosing $\mathcal{V}_1 =\cdots = \mathcal{V}_{n-1} = 0,\  \mathcal{R}$ satisfies
        \begin{align} \label{equation}
           ( \tilde{b}\delta-b^{\prime}) f(\mathcal{Y}_1,\ldots,\mathcal{Y}_{n-1},\mathcal{V}_n)f(\chi_1,\dots,\chi_n)+f(\chi_1,\dots,\chi_n)b^{\prime} f(\mathcal{Y}_1,\ldots,\mathcal{Y}_{n-1},\mathcal{V}_n)\nonumber \\
          =0
        \end{align}
    for all $\chi_1,\ldots,\chi_n,\mathcal{Y}_1,\ldots,\mathcal{Y}_{n-1}$ and $\mathcal{V}_n\in \mathcal{R}$. Now by using the Lemma $3$ of \cite{V. D. Filippis 2021} in Equation (\ref{equation}), we have $( \tilde{b}\delta-b^{\prime}) f(\mathcal{Y}_1,\ldots,\mathcal{Y}_{n-1},\mathcal{V}_n)=-b^{\prime} f(\mathcal{Y}_1,\ldots,\mathcal{Y}_{n-1},\mathcal{V}_n))\in C.$ Again  (Note $16$, \cite{V. D. Filippis 2021}) we have $( \tilde{b}\delta-b^{\prime}) =0=b^{\prime}$, which gives that both $\tilde{b}$ and $b^{\prime}$ are equals to zero, a contradiction.
    
On the other hand if the automorphism $\alpha $ is inner then there exists $t\in \mathcal{U}$ such that $\alpha(\chi) = t\chi t^{-1}$ for all $\chi\in \mathcal{R}$ and Equation (\ref{m20}) reduces to
\begin{align} \label{ed5}
        \bigg[\tilde{b}\delta t\sum\limits_{\sigma\in \mathcal{S}_n}\gamma_{\sigma}\sum\limits_{j=0}^{n-1}\chi_{\sigma(1)}\chi_{\sigma(2)}\ldots, \chi_{\sigma(j)}t^{-1} \mathcal{V}_{\sigma(j+1)}\chi_{\sigma(j+2)},\ldots,\chi_{\sigma(n)}\bigg]\nonumber\\ 
        f(\chi_1,\dots,\chi_n) +f(\chi_1,\dots,\chi_n) \nonumber\\ 
        \bigg[b^{\prime}t\sum\limits_{\sigma\in \mathcal{S}_n}\gamma_{\sigma}\sum\limits_{j=0}^{n-1}\chi_{\sigma(1)}\chi_{\sigma(2)}\ldots, \chi_{\sigma(j)}t^{-1} \mathcal{V}_{\sigma(j+1)}\chi_{\sigma(j+2)},\ldots,\chi_{\sigma(n)}\bigg]\nonumber\\ 
    -b^{\prime}tf(\chi_1,\dots,\chi_n)\nonumber\\ 
    \bigg[\sum\limits_{\sigma\in \mathcal{S}_n}\gamma_{\sigma}\sum\limits_{j=0}^{n-1}\chi_{\sigma(1)}\chi_{\sigma(2)}\ldots, \chi_{\sigma(j)}t^{-1} \mathcal{V}_{\sigma(j+1)}\chi_{\sigma(j+2)},\ldots,\chi_{\sigma(n)}\bigg] \nonumber \\
    -\bigg[b^{\prime}t\sum\limits_{\sigma\in \mathcal{S}_n}\gamma_{\sigma}\sum\limits_{j=0}^{n-1}\chi_{\sigma(1)}\chi_{\sigma(2)}\ldots, \chi_{\sigma(j)}t^{-1} \mathcal{V}_{\sigma(j+1)}\chi_{\sigma(j+2)},\ldots,\chi_{\sigma(n)}\bigg]\nonumber \\ f(\chi_1,\dots,\chi_n)     =0
    \end{align}
    for all $\chi_1,\ldots,\chi_n\in \mathcal{R}$. In particular, choosing $\mathcal{V}_1 =\cdots = \mathcal{V}_{n-1} =0$ and $\mathcal{V}_n = t\chi_n$, we get that
    \begin{align}
        (\tilde{b}\delta -2b^{\prime})tf(\chi_1,\dots,\chi_n)^2+f(\chi_1,\dots,\chi_n)b^{\prime}tf(\chi_1,\dots,\chi_n)=0
    \end{align}
    for all $\chi_1,\ldots,\chi_n\in \mathcal{R}$. Then from Fact $\ref{fac 8}$, we have $ (\tilde{b}\delta -2b^{\prime})t=-b^{\prime}t\in \mathcal{C}$ which gives that $\tilde{b}\delta t,b^{\prime}t\in \mathcal{C}$ infact $\tilde{b}\delta-b^{\prime}=0.$ Using these relations in Equation (\ref{ed7}) we have
     \begin{align}\label{ed8}
          (a-c+\tilde{b}\eta)f(\chi_1,\ldots,\chi_n)^2
        -\tilde{b}tf(\chi_1,\ldots,\chi_n)t^{-1}\eta f(\chi_1,\ldots,\chi_n)\nonumber \\
    +f(\chi_1,\ldots,\chi_n)cf(\chi_1,\ldots,\chi_n)
    =0
    \end{align}
    for all $\chi_1,\ldots,\chi_n\in \mathcal{R}$. Then from Lemma \ref{lem2} and Lemma \ref{lem3}, either $\tilde{b}t\in \mathcal{C}$ or $t^{-1}\eta\in \mathcal{C}$. If $\tilde{b}t\in \mathcal{C}$, Equation (\ref{ed8}) reduced to 
    \begin{align}\label{ed9}
          (a-c+\tilde{b}\eta)f(\chi_1,\ldots,\chi_n)^2
    +f(\chi_1,\ldots,\chi_n)(c-\tilde{b}\eta)f(\chi_1,\ldots,\chi_n)
    =0
    \end{align}
     for all $\chi_1,\ldots,\chi_n\in \mathcal{R}$. Then from Fact \ref{fac 8}, we have $ (a-c+\tilde{b}\eta)=-(c-\tilde{b}\eta)\in \mathcal{C}$ or $a=0$. Now putting $a=0$ and $(c-\tilde{b}\eta)\in \mathcal{C}$ in Equation (\ref{ed8}) we have
     \begin{align}
         (-c+\tilde{b}\eta)f(\chi_1,\ldots,\chi_n)^2
    +f(\chi_1,\ldots,\chi_n)(c-\tilde{b}\eta)f(\chi_1,\ldots,\chi_n)
    =0
     \end{align}
     for all $\chi_1,\ldots,\chi_n\in \mathcal{R}$. Then again from Fact \ref{fac 8}, we have $c-\tilde{b}\eta\in \mathcal{C}$ and this Conclusion $(5)$ of the main theorem.    
     
     Now if $t^{-1}\eta\in \mathcal{C}$ then $t\in \mathcal{C}$, which makes $\alpha$ identity automorphism, a contradiction to our assumption.
     
     
    	The following corollaries are immediate consequences of our main theorem.
	\begin{cor}
	Let $\mathcal{R}$ be a prime ring of characteristic different from $2$ with Utumi quotient ring $\mathcal{U}$ and extended centroid $\mathcal{C}$, $\phi(\chi_1,\ldots,\chi_n)$ be a multilinear polynomial over $\mathcal{C}$, which is not central valued on $\mathcal{R}$. Let $\mathcal{F}$ be a non-zero $\mathrm{b}$-generalized skew derivations on $\mathcal{R}$ acts like a Jordan derivations on $\mathcal{R}$,
	$$\mathcal{F}(\phi(\chi))\phi(\chi)+\phi(\chi) \mathcal{F}(\phi(\chi))=\mathcal{F}(\phi(\chi)^2).$$ 
	Then there exist $a,\tilde{a} \in \mathcal{U}$ such that $\mathcal{F}(\chi)= a\chi+\chi\tilde{a}$  for all $\chi \in \mathcal{R}$ with  $\tilde{a}+a=0$.
	\end{cor}
	\begin{proof}
	    If we take $\mathcal{F}=\mathcal{G}$ in Theorem \ref{thm}, we get the required conclusion. 
	\end{proof}
	\begin{cor}
	Let $\mathcal{R}$ be a prime ring of characteristic different from $2$ with Utumi quotient ring $\mathcal{U}$ and extended centroid $\mathcal{C}$, $\phi(\chi_1,\ldots,\chi_n)$ be a multilinear polynomial over $\mathcal{C}$, which is not central valued on $\mathcal{R}$. Let $\mathcal{G}$ be a $\mathrm{b}$-generalized skew derivations on $\mathcal{R}$ acts as a $2$-Jordan multiplier, $$\phi(\chi) \mathcal{G}(\phi(\chi))=\mathcal{G}(\phi(\chi)^2).$$ Then one of the following holds:
	\begin{enumerate}
	   \item  There exist $\tilde{c} \in \mathcal{C}$ and ${c} \in \mathcal{U}$ such that  $\mathcal{G}(\chi)=\tilde{c} \chi+\chi c$ for all $\chi \in \mathcal{R}$.
    \item There exist ${c} \in \mathcal{U}$ such that  $\mathcal{G}(\chi)= \chi c$ for all $\chi \in \mathcal{R}$.
	    \item There exist ${c} \in \mathcal{C}$ such that  $\mathcal{G}(\chi)={c} \chi$ for all $\chi \in \mathcal{R}$.
	\end{enumerate}
	\end{cor}
		\begin{proof}
	    If we take $\mathcal{F}=0$ in Theorem \ref{thm}, we get the required conclusion. 
	\end{proof}
   

\end{document}